\pdfoutput=1
\documentclass[final]{siamltex}%
\usepackage{amsfonts}
\usepackage{amssymb,amsmath,epsfig}
\usepackage{amsmath}
\usepackage{amssymb}
\usepackage{graphicx}%
\providecommand{\U}[1]{\protect\rule{.1in}{.1in}}
\newtheorem{example}[theorem]{Example}
\newtheorem{remark}[theorem]{Remark}
\begin{document}

\title{How to transform and filter images using iterated function systems.}
\author{Michael~F.~Barnsley \thanks{Department of Mathematics Australian National
University Canberra, ACT, Australia (\texttt{michael.barnsley@math.anu.edu.au}%
).}
\and Brendan Harding \thanks{Department of Mathematics, Australian National
University (\texttt{brendan.harding@anu.edu.au}).}
\and Konstantin~Igudesman\thanks{Faculty of Mechanics and Mathematics Kazan State
University Kazan, Russian Federation (\texttt{kigudesm@yandex.ru}).}}
\maketitle

\begin{abstract}
We present a general theory of fractal transformations and show how it leads
to new type of method for filtering and transforming digital images. This work
substantially generalizes earlier work on fractal tops. The approach involves
fractal geometry, chaotic dynamics, and an interplay between discrete and
continuous representations. The underlying mathematics is established and
applications to digital imaging are described and exemplified.

\end{abstract}

\begin{keywords}
Iterated function systems, dynamical systems, fractal transformations.
\end{keywords}

\begin{AMS}
37B10, 54H20, 68U10
\end{AMS}

\pagestyle{myheadings} \thispagestyle{plain} \markboth{M.~F.~BARNSLEY,
B.~HARDING, K.~IGUDESMAN}{FRACTAL MASKS}

\section{Introduction}

Fractal transformations are mappings between pairs of attractors of iterated
function systems. They are defined with the aid of code space structures, and
can be quite simple to handle and compute. They can be applied to digital
images when the attractors are rectangular subsets of $\mathbb{R}^{2}$. They
are termed "fractal" because they can change the box-counting, Hausdorff, and
other dimensions of sets and measures upon which they act. In this paper we
substantially generalize and develop the theory and we illustrate how it may
be applied to digital imaging. Previous work was restricted to fractal
transformations defined using fractal tops.

Fractal tops were introduced in \cite{barnsley} and further developed in
\cite{superfractals, monthly, germany, igudesman}. The main idea is this:
given an iterated function system with a coding map and an attractor, a
\textit{section} of the coding map, called a tops function, can be defined
using the "top" addresses of points on the attractor. Given two iterated
function systems each with an attractor, a coding map, and a common code
space, a mapping from one attractor to the other can be constructed by
composing the tops function, for the first iterated function system, with the
coding map for the second system. Under various conditions the composed map,
from one attractor to the other, is continuous or a homeomorphism. In the
cases of affine and projective iterated function systems, practical methods
based on the chaos game algorithm \cite{barnchaos} are feasible for the
approximate digital computation of such transformations. Fractal tops have
applications to information theory and to computer graphics. They have been
applied to the production of artwork, as discussed for example in
\cite{barnsleynotices}, and to real-time image synthesis \cite{nickei}. In the
present paper we extend the theory and applications.

Much of the material in this paper is new. The underlying new idea is that
diverse sections of a coding map may be defined quite generally, but
specifically enough to be useful, by associating certain dynamical systems
with the iterated function system. These sections provide novel collections of
fractal transformations; by their means we generalize the theory and
applications of fractal tops. We establish properties of fractal
transformations, including conditions under which they are continuous. The
properties are illustrated by examples related to digital imaging.

A notable result, Theorem \ref{goldthm}, states the existence of nontrivial
fractal homeomorphisms between attractors of some affine overlapping iterated
function systems. The proof explains how to construct them.\ An example of one
of these new homeomorphisms, applied to a picture of Lena, is illustrated in
Figure \ref{goldenlenna}.%
\begin{figure}[ptb]%
\centering
\includegraphics[
natheight=6.826900in,
natwidth=6.826900in,
height=3.4006in,
width=3.4006in
]%
{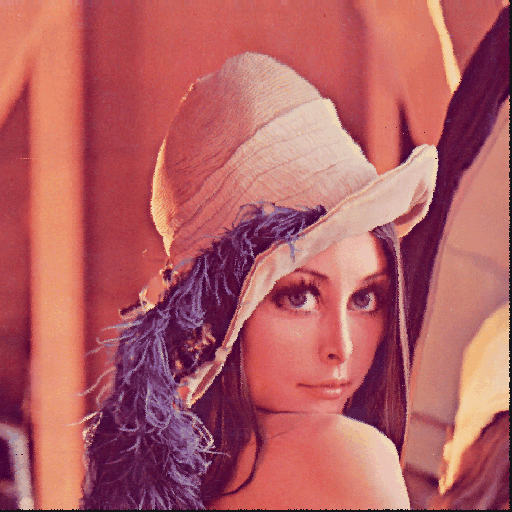}%
\caption{Lena after application of a fractal homeomorphism. See Example
\ref{goldenlennaex}.}%
\label{goldenlenna}%
\end{figure}

In Section \ref{ifssec} we review briefly the key definitions and results
concerning point-fibred iterated function systems on compact Hausdorff spaces.
Since this material is not well-known, it is of independent interest. The main
result is Theorem \ref{mainIFSthm}. This can be viewed as a restatement of
some ideas in \cite{kieninger}; it describes the relationship between the
coding map and the attractor of a point-fibred iterated function system.

In Section \ref{topsec} we define, and establish some general properties of,
fractal transformations constructed using sections of coding maps. In Theorem
\ref{sectionthm} we present some general properties of coding maps. Then we
use coding maps to define fractal transformations and, in Theorem
\ref{ctyfractaltransthm}, we provide sufficient conditions for a fractal
transformation to be continuous or homeomorphic.

In Section \ref{masksec} we define two different types of section of a coding
map: (i) with the aid of a \textit{masked} dynamical system; and (ii) with the
aid of fractal tops. Theorem \ref{maskthm} establishes the connection between
the masked dynamical system and a \textit{masked section} of the coding map.
Theorem \ref{maskbranchthm} includes a statement concerning the relationship
between the masked section to the coding map and the shift map. Here we also
establish the relationship between fractal tops and masked systems. A key
result, Theorem \ref{L:disjoint}, gives a condition under which the ranges of
different masked sections intersect in a set of measure zero. This enables the
approximate storage of multiple images in a single image, as illustrated in
Figure \ref{allmasked}.

In Section \ref{appsec} we apply and illustrate the theoretical structures of
Sections 2,3, and 4, in the context of digital imaging. Our goal is to
illustate the diversity of imaging techniques that are made feasible by our
techniques, and to suggest that fractal transformations have a potentially
valuable role to play in digital imaging. In Section \ref{appsec1} we
illustrate how fractal transformations may be applied to image synthesis, that
is to making artificial interesting and even beautiful pictures. Specifically
we explain how the technique of color-stealing \cite{barnsley} extends to
masked systems. In Section \ref{fhomsec} we apply fractal homeomorphisms,
using Theorem \ref{ctyfractaltransthm}, to transform digital images, for image
beautification, roughening, and special effects; in particular, we present and
illustrate Theorem \ref{goldthm} which extends the set of known affine fractal
homeomophisms. In Section \ref{filtsec} we consider the idea of composing a
fractal transformation, discretization, and the inverse of the transformation
to make idempotent image filters. In Section \ref{packsec} we apply Theorem
\ref{L:disjoint} to the approximate storage or encryption of multiple images
in a single image. In Section \ref{meassec} we provide a second technique for
combining several images in one: it combines invariant measures of a single
iterated function system with several sets of probabilities, to make a single
"encoded" image: approximations to the original images are revealed by the
application of several fractal homeomorphisms.

\section{\label{ifssec}Point-fibred iterated function systems}

Let $X$ be a nonempty compact Hausdorff space, and let $K(X)$ be the set of
nonempty compact subsets of $X$. It is known that $K(X)$ endowed with the
Vietoris topology is a compact Hausdorff space, see for example \cite[Theorem
2.3.5, p.17]{klein}. This encompasses the well-known fact that if $X$ is a
compact metric space then $K(X)$ endowed with the Hausdorff metric is a
compact metric space. Let $I=\left\{  1,2,...N\right\}  $ be a finite index
set with the discrete topology. Let $\left\{  f_{i}:X\rightarrow X|i\in
I\right\}  $ be a sequence of continuous functions. Following \cite{barnsley}%
$,$
\[
\mathcal{F}:=(X;f_{1},\ldots f_{N})
\]
is called an \textit{iterated function system} over $X$.

Following \cite[Definition 4.1.4, p.84]{kieninger} we define a map
\begin{equation}
\Pi:I^{\infty}\rightarrow K(X),\text{ }\sigma\mapsto\bigcap\limits_{k=1}%
^{\infty}f_{\sigma_{1}}\circ f_{\sigma_{2}}\circ\cdots\circ f_{\sigma_{k}}(X)
\label{codemapeq}%
\end{equation}
for all sequences $\sigma=\sigma_{1}\sigma_{2}\sigma_{3}...$ belonging to
$I^{\infty}$. The map is well-defined because $\Pi(\sigma)$ is the
intersection of a nested sequence of nonempty compact sets. The following
definition is based on \cite[Definition 4.3.6, p.97]{kieninger}.

\begin{definition}
Let $\mathcal{F}:=(X;f_{1},\ldots f_{N})$ be an iterated function system over
a compact Hausdorff space $X$. If $\Pi(\sigma)$ is a singleton for all
$\sigma\in I^{\infty}$ then $\mathcal{F}$ is said to be \textbf{point-fibred}%
$,$ and the \textbf{coding map} of $\mathcal{F}$ is defined by%
\[
\pi:I^{\infty}\rightarrow A\text{, }\left\{  \pi(\sigma)\right\}  =\Pi
(\sigma),
\]
where $A\subset X$ denotes the range of $\pi$.
\end{definition}

Theorem \ref{kieningerthm}, due to Kieninger, plays a central role in this
paper. It generalizes a classical result of Hutchinson \cite{hutchinson} that
applies when $X$ is a compact metric space and each $f\in\mathcal{F}$ is a contraction.

\begin{theorem}
\label{kieningerthm}Let $I^{\infty}$ have the product topology. If
$\mathcal{F}$ is a point-fibred iterated function system on a compact
Hausdorff space $X$ then the coding map $\pi:I^{\infty}\rightarrow A$ is continuous.
\end{theorem}

\begin{proof}
This follows from \cite[Proposition 4.3.22, p.105]{kieninger}.
\end{proof}

We define
\[
\mathcal{F}:K(X)\rightarrow K(X)\text{, }B\mapsto\bigcup_{f\in\mathcal{F}%
}f(B)\text{.}%
\]
By slight abuse of notation we use the same symbol $\mathcal{F}$ for the
iterated function system, the maps that it comprises, and the latter function.
We define $\mathcal{F}^{0}=i_{X}$, the identity map on $X$, and $\mathcal{F}%
^{k}=\mathcal{F}\circ\mathcal{F}^{k-1}$ for $k=1,2,...$ . The following
definition is a natural generalization of the notion of an attractor of a
contractive iterated function system, see for example \cite[definition on
p.1193 and Theorem 11.1, p.1206]{mcgehee}, \cite[p.107]{kieninger}, and also
\cite{barnchaos}.

\begin{definition}
\label{attractordef}Let $\mathcal{F}$ be an iterated function system on a
compact Hausdorff space $X$. An \textbf{attractor} of $\mathcal{F}$ is $A\in$
$K(X)$ with these properties: (i) $\mathcal{F}(A)=A$; (ii) there exists an
open set $U\subset X$ such that $A\subset U$ and
\begin{equation}
\lim_{k\rightarrow\infty}\mathcal{F}^{k}(B)=A \label{attractoreq}%
\end{equation}
for all $B\subset U$ with $B\in K(X)$. (The limit is with respect to the
Vietoris topology on $K(X)$.) The largest open set $\mathcal{B}\subset X$ such
that equation \ref{attractoreq} holds for all $B\subset\mathcal{B}$ with $B\in
K(X)$ is called the \textbf{basin} of $A$.
\end{definition}

The relationship between coding maps and attractors is provided by noting
that
\[
\pi(\sigma)=\lim_{k\rightarrow\infty}f_{\sigma|k}(a)\text{ where }f_{\sigma
|k}:=f_{\sigma_{1}}\circ f_{\sigma_{2}}\circ\cdots\circ f_{\sigma_{k}}%
\]
for $a\in X$. This leads to the following theorem.

\begin{theorem}
\label{mainIFSthm}If $\mathcal{F}$ is a point-fibred iterated function system
on a compact Hausdorff space $X$ then

(i) $\mathcal{F}:K(X)\rightarrow K(X)$ has a unique fixed-point $A\in K(X),$
i.e. $\mathcal{F}\left(  A\right)  =A$;

(ii) $A$ is the unique attractor of $\mathcal{F}$;

(iii) $A$ is equal to the range of the coding map $\pi$, namely%
\[
A=\pi(I^{\infty})\text{;}%
\]

(iv) the basin of $A$ is $X;$

(v) if $B\in K(X)$ then $\left\{  \pi(\sigma)\right\}  =\lim_{k\rightarrow
\infty}f_{\sigma|k}(B)$ for all $\sigma\in I^{\infty}.$
\end{theorem}

\begin{proof}
This follows from \cite[Proposition 4.4.2, p.107, see also Proposition 3.4.4,
p.77]{kieninger}.
\end{proof}

The following remark tells us that if an iterated function system possesses an
attractor then it is point-fibred when restricted to a certain neighborhood of
the attractor. Necessary and sufficient conditions for an iterated function
system of projective transformations to possess an attractor are given in
\cite{BVW}.

\begin{remark}
Let $A$ be an attractor of an iterated function system $\mathcal{F}$ on a
compact Hausdorff space $X$, and let $\mathcal{B}$ be the basin of $A.$
Then, following \cite{mcgehee}, $\mathcal{F}$ defines an iterated closed
relation $r:=\{(x,f(x)):x\in X,f\in \mathcal{F}\}\subset X\times X$, and $A$
is an attractor of $r$. By \cite[Theorem 7.2, p.1193]{mcgehee} there exists
a compact neighborhood $V$ of $A$ such that $A\subset Int_{X}(V)$ and $\mathcal{F%
}\left( V\right) \subset Int_{X}(V).$ ($Int_{X}(S)$ denotes the interior of the set $%
S\subset X$ in the subspace topology induced on $S$ by $X$.) It follows that $\mathcal{F}%
|_{V}:=(V;f_{1}|_{V},f_{2}|_{V},...f_{N}|_{V})$ is a point-fibred iterated
function system on a compact Hausdorff space. The attractor of $\mathcal{F}%
|_{V}$ is $A.$
\end{remark}

\begin{definition}
Let $\mathcal{F}$ be a point-fibred iterated function system on a compact
Hausdorff space. The set $I^{\infty}$ is called the \textbf{code space}
of\ $\mathcal{F}$. A point $\sigma\in I^{\infty}$ is called an
\textbf{address} of $\pi(\sigma)\in A$.
\end{definition}

In the rest of this paper the underlying space $X$ is a compact Hausdorff
space. Also in the rest of this paper the symbols $\mathcal{F},\mathcal{G}%
,\mathcal{H}$ denote point-fibred iterated function systems on compact
Hausdorff spaces. We will say that an iterated function system is
\textit{injective} when all of the maps that it comprises are injective. We
will say that an iterated function system is \textit{open} when all of the
maps that it comprises are open.

\section{\label{topsec}Fractal transformations}

Here we present a generalized theory of fractal transformations and establish
some continuity properties. Fractal transformations are defined using sections
of coding maps. We are concerned with continuity properties; for example,
Theorem \ref{ctyfractaltransthm} (i) provides a sufficient condition for a
fractal transformation to be continuous.

\begin{definition}
\label{branchdef}Let $\pi:I^{\infty}\rightarrow A$ be the coding map of
$\mathcal{F}$. A subset $\Omega\subset I^{\infty}$ is called an
\textbf{address space }for $\mathcal{F}$ if $\pi(\Omega)=A$ and $\pi|_{\Omega
}:\Omega\rightarrow A\subset X$ is one-to-one. The corresponding map
\[
\tau:A\rightarrow\Omega,x\mapsto(\pi|_{\Omega})^{-1}(x)\text{,}%
\]
is called a \textbf{section }of\textbf{\ }$\pi$.
\end{definition}

Theorem \ref{sectionthm} summarises the properties of sections of $\pi.$

\begin{theorem}
\label{sectionthm}Let $\mathcal{F}=(X;f_{1},f_{2},...,f_{N})$ be a
point-fibred iterated function system on a compact Hausdorff space, with
attractor $A$, code space $I^{\infty}$, and coding map $\pi:I^{\infty
}\rightarrow A.$ If $\tau:A\rightarrow\Omega$ is a section of $\pi$ then

(i) $\tau:A\rightarrow\Omega$ is bijective;

(ii) $\tau^{-1}:\Omega\rightarrow$ $A$ is continuous;

(iii) $\pi\circ\tau=\iota_{A}$, the identity map on $A,$ and $\tau\circ\left(
\pi|_{\Omega}\right)  =\iota_{\Omega}$, the identity map on $\Omega$;

(iv) if $\mathcal{F}$ is injective and $f_{i}(A)\cap f_{j}(A)=\emptyset$ for
all$\ i,j\in I$ with $i\neq j$, then $\Omega=I^{\infty}$;

(v) if $\Omega$ is closed then $\tau:A\rightarrow\Omega$ is a homeomorphism;

(vi) if $A$ is connected and $A$ is not a singleton, then $\tau:A\rightarrow
\Omega$ is not continuous.
\end{theorem}

\begin{proof}
(i) By Definition \ref{branchdef} $\pi|_{\Omega}:\Omega\rightarrow A$ is
bijective, so $\tau=$ $(\pi|_{\Omega})^{-1}:A\rightarrow\Omega$ is bijective.

(ii) By Theorem \ref{kieningerthm} $\pi:I^{\infty}\rightarrow A$ is
continuous. It follows that $\tau^{-1}=\pi|_{\Omega}:\Omega\rightarrow A$ is continuous.

(iii) If $x\in A$ then $\pi\circ\tau(x)=\pi\circ\left(  \pi|_{\Omega}\right)
^{-1}(x)\subset\pi\circ\pi^{-1}(x)=x=i_{A}(x).$ Also $\tau\circ\pi|_{\Omega
}=\tau\circ\tau^{-1}=\iota_{\Omega}$.

(iv) Suppose $\Omega\neq I^{\infty}$. Then there are $\sigma,\omega\in
I^{\infty}$, $\sigma\neq\omega$, such that $\pi(\sigma)=\pi(\omega)$. We show
that this is impossible. If $\pi(\sigma)=\pi(\omega)$ then Theorem
\ref{mainIFSthm} (v) implies $\left\{  \pi(\sigma)\right\}  =\lim
_{k\rightarrow\infty}f_{\sigma|k}(A)=\lim_{k\rightarrow\infty}f_{\omega
|k}(A)=\left\{  \pi(\omega)\right\}  $. Let $K$ be the least integer such that
$\omega_{K+1}\neq\sigma_{K+1}$. Then $f_{\omega|K}=f_{\sigma|K}$ and
$f_{\omega_{K+1}}(A)\cap f_{\sigma_{K+1}}(A)=\emptyset$. Since $\mathcal{F}$
is injective, each $f\in\mathcal{F}$ is injective, which implies that
$f_{\omega|K}:X\rightarrow X$ is injective. Since $f_{\omega|K}=f_{\sigma|K}$,
it now follows that $f_{\omega|K+1}(A)\cap f_{\sigma|K+1}(A)=f_{\omega
|K}(f_{\omega_{K+1}}\left(  A\right)  )\cap f_{\sigma|K}(f_{\sigma_{K+1}%
}\left(  A\right)  )=f_{\omega|K}(f_{\sigma_{K+1}}\left(  A\right)  )\cap
f_{\omega|K}(f_{\omega_{K+1}}\left(  A\right)  )=\emptyset$. Since $\left\{
\pi(\sigma)\right\}  \subset f_{\sigma|K+1}(A)$ and $\left\{  \pi
(\omega)\right\}  \subset f_{\omega|K+1}(A)$ it now follows that $\left\{
\pi(\sigma)\right\}  \cap\left\{  \pi(\omega)\right\}  =\emptyset$.

(v) If $\Omega\subset I^{\infty}$ is closed then it is compact, because
$I^{\infty}$ is compact. It follows that $\tau^{-1}=\pi|_{\Omega}%
:\Omega\rightarrow A$ is a continuous bijective mapping from a compact space
$\Omega$ onto a Hausdorff space $A$. By \cite[Theorem 5.6, p.167]{munkres} it
follows that $\tau^{-1}:\Omega\rightarrow A$ is a homeomorphism. It follows
that $\tau:A\rightarrow\Omega$ is a homeomorphism.

(vi) Suppose that $A$ is connected and $A$ is not a singleton. It follows that
$\Omega$ is not a singleton. It follows that $\Omega$ is not connected. (Since
$\Omega$ contains more than one point and is a subset of $I^{\infty},$ which
is totally disconnected when it contains more than one point, it follows that
$\Omega$ is not connected.)

Now suppose $\tau:A\rightarrow\Omega$ is continuous. Then $\tau$ is a
homeomorphism. It follows that $A$ is not connected. But $A$ is connected. So
$\tau:A\rightarrow\Omega$ is not continuous.
\end{proof}

In Definition \ref{ftransdef} we define a type of transformation between
attractors of iterated function systems by composing sections of coding maps
with coding maps. We call these transformations "fractal" because they can be
very rough; specific examples demonstrate that the graphs of these
transformations, between compact manifolds, can possess a non-integer
Hausdorff-Besicovitch dimension.

\begin{definition}
\label{ftransdef}Let $\mathcal{F}=(X;f_{1},f_{2},...,f_{N})$ be a point-fibred
iterated function system over a compact Hausdorff space $X$. Let
$A_{\mathcal{F}}\subset X$ be the attractor of $\mathcal{F}$. Let
$\pi_{\mathcal{F}}:I^{\infty}\rightarrow A_{\mathcal{F}}$ be the coding map of
$\mathcal{F}$. Let $\tau_{\mathcal{F}}:A_{\mathcal{F}}\rightarrow
\Omega_{\mathcal{F}}\subset I^{\infty}$ be a branch of $\pi_{\mathcal{F}}$.
Let $\mathcal{G}=(Y;g_{1},g_{2},...,g_{N})$ be a point-fibred iterated
function system over a compact Hausdorff space $Y.$ Let $A_{\mathcal{G}}$ be
the attractor of $\mathcal{G}$. Let $\pi_{\mathcal{G}}:I^{\infty}\rightarrow
A_{\mathcal{G}}$ be the coding map of $\mathcal{G}$. The corresponding
\textbf{fractal transformation} is defined to be%
\[
T_{\mathcal{FG}}:A_{\mathcal{F}}\rightarrow A_{\mathcal{G}}\text{,
}x\longmapsto\pi_{\mathcal{G}}\circ\tau_{\mathcal{F}}(x)\text{.}%
\]

\end{definition}

In Theorem \ref{ctyfractaltransthm} we describe some continuity properties of
fractal transformations. These properties make fractal transformations
interesting for applications to digital imaging.

\begin{theorem}
\label{ctyfractaltransthm}Let $\mathcal{F}$ and $\mathcal{G}$ be point-fibred
iterated function systems as in Definition \ref{ftransdef}. Let
$T_{\mathcal{FG}}=\pi_{\mathcal{G}}\circ\tau_{\mathcal{F}}:A_{\mathcal{F}%
}\rightarrow A_{\mathcal{G}}$ be the corresponding fractal transformation.

(i) If, whenever $\sigma,\omega\in\overline{\Omega_{\mathcal{F}}}$,
$\pi_{\mathcal{F}}(\sigma)=\pi_{\mathcal{F}}(\omega)$ $\Rightarrow$
$\pi_{\mathcal{G}}(\sigma)=\pi_{\mathcal{G}}(\omega),$ then $T_{\mathcal{FG}%
}:A_{\mathcal{F}}\rightarrow A_{\mathcal{G}}$ is continuous.

(ii) If $\Omega_{\mathcal{G}}:=\Omega_{\mathcal{F}}$ is an address space for
$\mathcal{G}$, and if, whenever $\sigma,\omega\in\overline{\Omega
_{\mathcal{F}}}$, $\pi_{\mathcal{F}}(\sigma)=\pi_{\mathcal{F}}(\omega)$
$\Longleftrightarrow$ $\pi_{\mathcal{G}}(\sigma)=\pi_{\mathcal{G}}(\omega),$
then $T_{\mathcal{FG}}$ is a homeomorphism and $T_{\mathcal{FG}}%
^{-1}=T_{\mathcal{GF}}$.
\end{theorem}

\begin{proof}
(i) Assume that, whenever $\sigma,\omega\in\overline{\Omega_{\mathcal{F}}}$,
$\pi_{\mathcal{F}}(\sigma)=\pi_{\mathcal{F}}(\omega)$ $\Rightarrow$
$\pi_{\mathcal{G}}(\sigma)=\pi_{\mathcal{G}}(\omega)$. We begin by showing
that $\pi_{\mathcal{G}}\circ\tau_{\mathcal{F}}\circ\pi_{\mathcal{F}%
}|_{\overline{\Omega_{\mathcal{F}}}}:\overline{\Omega_{\mathcal{F}}%
}\rightarrow A_{\mathcal{G}}$ is the same as $\pi_{\mathcal{G}}|_{\overline
{\Omega_{\mathcal{F}}}}:\overline{\Omega_{\mathcal{F}}}\rightarrow
A_{\mathcal{G}}.$ Let $\sigma\in\overline{\Omega_{\mathcal{F}}}$. Then there
is $\omega\in\Omega_{\mathcal{F}}$ such that $\pi_{\mathcal{F}}(\sigma
)=\pi_{\mathcal{F}}(\omega)$ because $\Omega_{\mathcal{F}}$ is a code space
for $\mathcal{F}.$ Hence $\tau_{\mathcal{F}}\circ\pi_{\mathcal{F}}%
|_{\overline{\Omega_{\mathcal{F}}}}(\sigma)=\tau_{\mathcal{F}}\circ
\pi_{\mathcal{F}}|_{\Omega_{\mathcal{F}}}(\omega)=\omega$ (using Theorem
\ref{sectionthm} (iii)). Hence $\pi_{\mathcal{G}}\circ\tau_{\mathcal{F}}%
\circ\pi_{\mathcal{F}}|_{\overline{\Omega_{\mathcal{F}}}}(\sigma
)=\pi_{\mathcal{G}}(\omega)=\pi_{\mathcal{G}}(\sigma)$, where we have used our
initial assumption. It follows that $\pi_{\mathcal{G}}\circ\tau_{\mathcal{F}%
}\circ\pi_{\mathcal{F}}|_{\overline{\Omega_{\mathcal{F}}}}=\pi_{\mathcal{G}%
}|_{\overline{\Omega_{\mathcal{F}}}}$. It follows that $\pi_{\mathcal{G}}%
\circ\tau_{\mathcal{F}}\circ\pi_{\mathcal{F}}|_{\overline{\Omega_{\mathcal{F}%
}}}$ is a continuous map from a compact space $\overline{\Omega_{\mathcal{F}}%
}$ to a compact Hausdorff space $A_{\mathcal{G}}.$ But $\pi_{\mathcal{G}}%
\circ\tau_{\mathcal{F}}\circ\pi_{\mathcal{F}}|_{\overline{\Omega_{\mathcal{F}%
}}}=\left(  \pi_{\mathcal{G}}\circ\tau_{\mathcal{F}}\right)  \circ
\pi_{\mathcal{F}}|_{\overline{\Omega_{\mathcal{F}}}}$ is the composition of a
continuous mapping $\pi_{\mathcal{F}}|_{\overline{\Omega_{\mathcal{F}}}%
}:\overline{\Omega_{\mathcal{F}}}\rightarrow A_{\mathcal{F}},$ from a compact
Hausdorff space $\overline{\Omega_{\mathcal{F}}}$ onto a Hausdorff space
$A_{\mathcal{F}}$, with a mapping $\pi_{\mathcal{G}}\circ\tau_{\mathcal{F}%
}:A_{\mathcal{F}}\rightarrow A_{\mathcal{G}}$ from $A_{\mathcal{F}}$ into a
Hausdorff space $A_{\mathcal{G}}$. It follows by a well-known theorem in
topology, see for example \cite[Proposition 7.4, p. 195]{mendelson}$,$ that
$T_{\mathcal{FG}}=$ $\pi_{\mathcal{G}}\circ\tau_{\mathcal{F}}:A_{\mathcal{F}%
}\rightarrow A_{\mathcal{G}}$ is continuous.

(ii) Assume that $\Omega_{\mathcal{G}}=\Omega_{\mathcal{F}}$ is an address
space for $\mathcal{G}$, and that, whenever $\sigma,\omega\in\overline
{\Omega_{\mathcal{F}}}$, $\pi_{\mathcal{F}}(\sigma)=\pi_{\mathcal{F}}(\omega)$
$\Longleftrightarrow$ $\pi_{\mathcal{G}}(\sigma)=\pi_{\mathcal{G}}(\omega)$.
Then by (i) both of the mappings $T_{\mathcal{FG}}=\pi_{\mathcal{G}}\circ
\tau_{\mathcal{F}}:A_{\mathcal{F}}\rightarrow A_{\mathcal{G}}$ and
$T_{\mathcal{GF}}=\pi_{\mathcal{F}}\circ\tau_{\mathcal{G}}:A_{\mathcal{G}%
}\rightarrow A_{\mathcal{F}}$ are continuous. Using the fact that the range of
$\tau_{\mathcal{G}}$ is $\Omega_{\mathcal{F}}$ it is readily checked that
$T_{\mathcal{FG}}\circ T_{\mathcal{GF}}=\pi_{\mathcal{G}}\circ\tau
_{\mathcal{F}}\circ$ $\pi_{\mathcal{F}}\circ\tau_{\mathcal{G}}%
=i_{A_{\mathcal{G}}}$ and $T_{\mathcal{GF}}\circ T_{\mathcal{FG}}%
=\pi_{\mathcal{F}}\circ\tau_{\mathcal{G}}\circ$ $\pi_{\mathcal{G}}\circ
\tau_{\mathcal{F}}=i_{A_{\mathcal{F}}}.$ Hence $T_{\mathcal{FG}}$ is a
homeomorphism and $T_{\mathcal{GF}}=\left(  T_{\mathcal{FG}}\right)  ^{-1}$.
\end{proof}

\begin{remark}
Suppose that the equivalence relations $\sim _{\mathcal{F}}$ and $\sim _{%
\mathcal{G}}$on $I^{\infty }$, induced by $\pi _{\mathcal{F}}:I^{\infty
}\rightarrow A_{\mathcal{G}}$ and $\pi _{\mathcal{G}}:I^{\infty }\rightarrow
A_{\mathcal{G}}$ respectively, are the same. Then it is well known that,
using compactness, the quotient topological space $I^{\infty }/\sim _{%
\mathcal{F}}=I^{\infty }/\sim _{\mathcal{G}}$ is homeomorphic to both $A_{%
\mathcal{F}}$ and $A_{\mathcal{G}}$.
\end{remark}

See \cite{bandt} for discussion of relationships between the topology of an
attractor and the equivalence class structure induced by a coding map.

\section{\label{masksec}Construction of sections of $\pi$}

In order to construct a fractal transformation we need to specify a section of
$\pi$. In order to construct a section of $\pi$ we have to construct an
address space $\Omega$; that is, we need to specify one element from each of
the sets in the collection $\{\pi^{-1}(x):x\in A\}$. To do this in a general
way seems to be difficult; for example, if $f_{1}(A)\cap f_{2}(A)$ contains a
nonempty open set $O$, then $\pi^{-1}(x)$ is non-denumerable for all $x\in O$.
However there are two particular related methods. These methods yield
interesting structure; for example, in both cases the resulting address space
$\Omega\subset I^{\infty}$ is mapped into itself by the shift operator, see
Theorem \ref{maskbranchthm} (ii). They are as follows.

(a) (Masked iterated function system method.) This method requires that
$\mathcal{F}$ is injective. Define a dynamical system on $T:A_{\mathcal{F}%
}\rightarrow A_{\mathcal{F}}$ with the aid of inverses of the functions in
$\mathcal{F}$. Follow orbits of $T$ to define $\tau_{\mathcal{F}%
}:A_{\mathcal{F}}\rightarrow\Omega_{\mathcal{F}}$; in effect one uses a Markov
partition associated with $T$ to define $\Omega_{\mathcal{F}}$.

(b) (Fractal tops method.) Use the dictionary order relation on $I^{\infty}$
to select a unique element of $\pi^{-1}(x)$ for each $x\in A_{\mathcal{F}}$.
This method applies when $\mathcal{F}$ is not required to be injective. When
$\mathcal{F}$ is injective, it is a special case of (a). To date, the fractal
tops method seems to be the easiest to convert to computational algorithms and applications.

\subsection{(a) The masked iterated function system method}

Definition \ref{maskdef} introduces a special partition of an attractor.

\begin{definition}
\label{maskdef}Let $\mathcal{F}$ be a point-fibred iterated function system on
a compact Hausdorff space. Let $A$ be the attractor of $\mathcal{F}$. A finite
sequence of sets $\mathcal{M}:=\{M_{i}\subset A|i\in I\}$ is called a
\textbf{mask} for $\mathcal{F}$ if

\begin{enumerate}
\item $M_{i}\subseteq f_{i}(A)$, $i\in I$;

\item $M_{i}\cap M_{j}=\emptyset$, $i,j\in I$, $i\neq j$;

\item $\cup_{i\in I}M_{i}=A$.
\end{enumerate}
\end{definition}

Note that for any $x\in A$ there exists a unique $i\in I$ such that $x\in
M_{i}\subseteq f_{i}(A)$. This enable us, in Definition \ref{dynsysdef}, to
define an associated dynamical system on the attractor.

\begin{definition}
\label{dynsysdef} Let $\left\{  M_{i}:i\in I\right\}  $ be a mask for an
injective point-fibred iterated function system $\mathcal{F}$ with attractor
$A.$ The associated \textbf{masked dynamical system} for $\mathcal{F}$ is
\[
T:A\rightarrow A,\quad x\mapsto%
\begin{cases}
f_{1}^{-1}(x), & x\in M_{1},\\
f_{2}^{-1}(x), & x\in M_{2},\\
\qquad\vdots & \\
f_{N}^{-1}(x), & x\in M_{N}.
\end{cases}
\]

\end{definition}

Theorem \ref{maskthm} associates a unique section of $\pi$ with a masked
dynamical system.

\begin{theorem}
\label{maskthm}Let $\mathcal{F}$ be an injective point-fibred iterated
function system with attractor $A.$ Let $T:A\rightarrow A$ be a masked
dynamical system for $\mathcal{F}$, associated with mask $\mathcal{M}=\left\{
M_{i}:i\in I\right\}  $. Let $x\in A$ and let $\left\{  x_{n}\right\}
_{n=0}^{\infty}$ be the orbit of $x$ under $T$; that is, $x_{0}=x$ and
$x_{n}=T^{n}(x_{0})$ for $n=1,2,...$. Let $\sigma_{k}(x)\in I$ be the unique
symbol such that $x_{k-1}\in M_{\sigma_{k}}$, for $k=1,2,3,...$ . Then
\[
\Omega_{\mathcal{M}}=\{\sigma\in I^{\infty}|\sigma:=\sigma(x)=\sigma
_{1}(x)\sigma_{2}(x)\sigma_{3}(x)...\in I^{\infty},x\in A\}
\]
is an address space for $\mathcal{F}$.
\end{theorem}

\begin{proof}
Let $x\in A$. We begin by proving that, for all $K\geq1$ and $j=1,2,...,K-1, $
we have%
\begin{equation}
x_{K-j}\in f_{\sigma_{K-j+1}(x)}f_{\sigma_{K-j+1}(x)}\circ...\circ
f_{\sigma_{K}(x)}(A). \label{inductioneq}%
\end{equation}
Fix $K.$ We use induction on $j.$ Since $x_{K-1}\in M_{\sigma_{K}}$ it follows
that $x_{K}=f_{\sigma_{K}}^{-1}(x_{K-1})$ so $x_{K-1}=f_{\sigma_{K}(x)}%
(x_{K})\in f_{\sigma_{K}}(A)$. It follows that equation \ref{inductioneq} is
true for $j=1$. Suppose that equation \ref{inductioneq} is true for
$j=1,2,...,J\leq K-2$. It follows that $x_{K-J}\in f_{\sigma_{K-J+1}(x)}\circ
f_{\sigma_{K-J+1}(x)}\circ...\circ f_{\sigma_{K}(x)}(A)$. We also have
$x_{K-J}\in M_{\sigma_{K-J+1}}$ so $x_{K-J}=f_{\sigma_{K-J}}^{-1}(x_{K-J-1})$
which implies $x_{K-J-1}=f_{\sigma_{K-J}}(x_{K-J})\in f_{\sigma_{K-J(x)}%
}\left(  f_{\sigma_{K-J+1}(x)}\circ f_{\sigma_{K-J+1}(x)}\circ...\circ
f_{\sigma_{K}(x)}(A)\right)  .$ Hence equation \ref{inductioneq} is true for
$j=J+1$. This completes the induction on $j$.

It follows that
\[
x=x_{0}\in f_{\sigma_{1}(x)}\circ f_{\sigma_{2}(x)}\circ...\circ f_{\sigma
_{K}(x)}(A)
\]
for all $K.$ It follows that $x\in\cap_{k=1}^{\infty}f_{\sigma_{1}(x)}\circ
f_{\sigma_{2}(x)}\circ...\circ f_{\sigma_{k}(x)}(A)=\pi(\sigma(x))$. It
follows that $x=\pi(\sigma(x))$ for all $x\in A.$ It follows that
$\pi|_{\Omega_{\mathcal{M}}}(\Omega_{\mathcal{M}})=A$; that is, $\pi
|_{\Omega_{\mathcal{M}}}:\Omega_{\mathcal{M}}\rightarrow A$ is surjective.

To show that $\pi|_{\Omega_{\mathcal{M}}}:\Omega_{\mathcal{M}}\rightarrow A$
is injective, suppose $\sigma(x)\neq\sigma(y)\in\Omega_{\mathcal{M}}$ for some
$x,y\in A.$ Then for some $k$ we have $\sigma(x)_{k}\neq\sigma(y)_{k}$ which
implies $M_{\sigma(x)_{k}}\cap M_{\sigma(y)_{k}}=\emptyset$. Hence
$T^{k}(x)\neq T^{k}(y).$ It follows that $x\neq y$.
\end{proof}

\begin{definition}
Let $\mathcal{F}$ be an injective point-fibred iterated function system. The
address space $\Omega_{\mathcal{M}}\subset I^{\infty}$ provided by Theorem
\ref{maskthm} is called a \textbf{masked address space} for $\mathcal{F}$. The
corresponding section of $\pi$, say $\tau:A\rightarrow\Omega_{\mathcal{M}}$,
is called a \textbf{masked section} of $\pi.$
\end{definition}

Masked address spaces and masked sections have all of the properies of address
spaces and sections, such as those in Theorem \ref{sectionthm}, and associated
fractal transformations have the properties in Theorem
\ref{ctyfractaltransthm}. But these objects have additional properties that
derive from the existence and structure of the masked dynamical system. Some
of these additional properties are described in Theorem \ref{maskbranchthm}.

\begin{theorem}
\label{maskbranchthm}If $\mathcal{F}=(X;f_{1},f_{2},...,f_{N})$ is an
injective point-fibred iterated function system, with attractor $A$, code
space $I^{\infty}$, coding map $\pi:I^{\infty}\rightarrow A$, mask
$\mathcal{M}=\{M_{i}:i\in I\}$, masked address space $\Omega_{\mathcal{M}}$
and masked section $\tau:A\rightarrow\Omega_{\mathcal{M}}$, then

(i) if $\mathcal{F}$ is open then $\tau:A\rightarrow\Omega_{\mathcal{M}}$ is
continuous at $x\in A$ if and only if $T^{k-1}(x)\in Int_{A}(M_{\tau(x)_{k}})$
for all $k=1,2,...$ ;

(ii) the shift map
\[
S:\Omega_{\mathcal{M}}\rightarrow\Omega_{\mathcal{M}},\sigma_{1}\sigma
_{2}\sigma_{3}...\mapsto\sigma_{2}\sigma_{3}\sigma_{4}...
\]
is well-defined, with $S(\Omega_{\mathcal{M}})\subset\Omega_{\mathcal{M}}$;

(iii) the following diagram commutes
\begin{equation}%
\begin{array}
[c]{ccc}%
A & \overset{T}{\rightarrow} & A\\
\tau\downarrow\text{\ \ \ \ } &  & \text{ \ \ \ }\downarrow\tau\\
\Omega_{\mathcal{M}} & \overset{S}{\rightarrow} & \Omega_{\mathcal{M}}%
\end{array}
\label{commutediagram}%
\end{equation}

(iv) if there is $i\in I$ such that $M_{i}=f_{i}(A)$ then $S(\Omega
_{\mathcal{M}})=\Omega_{\mathcal{M}}.$
\end{theorem}

\begin{proof}
(i) Suppose that $T^{k-1}(x)\in Int_{A}(M_{\tau(x)_{k}})$ for all $k=1,2,... $
. Let $\left\{  x^{(n)}\right\}  $ converge to $x.$ Let $\tau(x)=\sigma
=\sigma_{1}\sigma_{2}\sigma_{3}...$ and let $\tau\left(  x^{(n)}\right)
=\sigma^{\left(  n\right)  }=\sigma_{1}^{\left(  n\right)  }\sigma
_{2}^{\left(  n\right)  }\sigma_{3}^{\left(  n\right)  }...$. Since $x\in
Int_{A}(M_{\sigma_{1}})$ there is an integer $n_{1}$ such that $x^{(n)}\in
M_{\sigma_{1}}$ for all $n\geq n_{1}$. Since, for all $i\in I$, $f_{i}^{-1}$
is continuous (because $\mathcal{F}$ is invertible and open) we have $\left\{
T(x^{(n)})=f_{\sigma_{1}}(x^{\left(  n\right)  })\right\}  _{n=n_{1}}^{\infty
}$converges to $T(x)=f_{\sigma_{1}}(x)$ as $n\rightarrow\infty$. Similarly,
for any given $K$, there is $n_{K}$ such that $T^{p}(x^{(n)})\in M_{\sigma
_{p}}$ for all $p\leq K$ and all $n\geq n_{K}$. It follows that, for any given
$K$, $\sigma_{p}^{(n)}=\sigma_{p}$ for all $p\leq K$ and all $n\geq n_{K}$. It
follows that $\left\{  \sigma^{\left(  n\right)  }\right\}  _{n=1}^{\infty}$
converges to $\sigma$, i.e. $\left\{  \tau\left(  x^{(n)}\right)  \right\}
_{n=1}^{\infty}$ converges to $\tau(x). $ It follows that $\tau:A\rightarrow
\Omega_{\mathcal{M}}$ is continuous at $x$.

To prove the converse we assume that it is not true that $T^{k-1}(x)\in
Int_{A}(M_{\tau(x)_{k}})$ for all $k=1,2,...$ . It follows that there is some
$K\geq0$ such that $T^{K}(x)\in M_{\sigma_{K+1}}\backslash Int_{A}%
(M_{\sigma_{K+1}}).$ Here as in the first part of the proof, we write
$\tau(x)=\sigma=\sigma_{1}\sigma_{2}\sigma_{3}...$. It follows that there is a
sequence $\left\{  y_{j}\right\}  _{j=1}^{\infty}$ that converges to
$T^{K}(x)$ with $y_{j}\notin M_{\sigma_{K+1}}$ for all $j$. (Any neighborhood
of $T^{K}(x)$ must contain a point that is in $A\backslash M_{\sigma_{K+1}}$.)
It follows that $\left\{  f_{\sigma_{1}}\circ f_{\sigma_{2}}\circ...\circ
f_{\sigma_{K}}(y_{j})\right\}  _{j=1}^{\infty}$ converges to $f_{\sigma_{1}%
}\circ f_{\sigma_{2}}\circ...\circ f_{\sigma_{K}}(T^{K}(x))$ since
$f_{\sigma_{1}}\circ f_{\sigma_{2}}\circ...\circ f_{\sigma_{K}}:A\rightarrow
A$ is continuous. (We define $f_{\sigma_{1}}\circ f_{\sigma_{2}}\circ...\circ
f_{\sigma_{0}}=i_{A}$.) But $f_{\sigma_{1}}\circ f_{\sigma_{2}}\circ...\circ
f_{\sigma_{K}}(T^{K}(x))=x,$ while $\tau(f_{\sigma_{1}}\circ f_{\sigma_{2}%
}\circ...\circ f_{\sigma_{K}}(y_{j}))_{K+1}\neq\sigma_{K+1}$ because
$T^{K}(f_{\sigma_{1}}\circ f_{\sigma_{2}}\circ...\circ f_{\sigma_{K}}%
(y_{j}))=y_{j}\notin M_{\sigma_{K+1}}$. It follows that $\tau:A\rightarrow
\Omega_{\mathcal{M}}$ is not continuous. The desired conclusion follows at once.

(ii)\&(iii) Let $\sigma_{1}\sigma_{2}\sigma_{3}...\in\Omega_{\mathcal{M}}.$
Let $x=\tau^{-1}\left(  \sigma_{1}\sigma_{2}\sigma_{3}...\right)  $. Then
$\tau\circ T\circ\tau^{-1}\left(  \sigma_{1}\sigma_{2}\sigma_{3}...\right)
=\sigma_{2}\sigma_{3}\sigma_{4}...$, whence $\sigma_{2}\sigma_{3}\sigma
_{4}...\in\Omega_{\mathcal{M}}$. It follows that the shift map
\[
S:\Omega_{\mathcal{M}}\rightarrow\Omega_{\mathcal{M}},\sigma_{1}\sigma
_{2}\sigma_{3}...\mapsto\sigma_{2}\sigma_{3}\sigma_{4}...
\]
is well-defined, with $S(\Omega_{\mathcal{M}})\subset\Omega_{\mathcal{M}}$.

(iv) Let $\sigma_{1}\sigma_{2}\sigma_{3}...\in\Omega_{\mathcal{M}}$ and let
$i\in I$ be such that $M_{i}=f_{i}(A)$. Then $A\supset T(A)\supset T(M_{i})=A$
whence $A=T(A)$. By (iii)\ we have $S\circ\tau(A)=\tau\circ T(A)$ so $S\left(
\Omega_{\mathcal{M}}\right)  =\tau(A)=\Omega_{\mathcal{M}}$.
\end{proof}

We remark that masked dynamical systems are related to Markov partitions in
the theory of dynamical systems. See for example \cite[Proposition 18.7.8,
p.595]{katok}.

In general a masked dynamical system $T:A\rightarrow A$ depends on the mask
$\mathcal{M}$. By suitable choice of mask we can sometimes obtain a dynamical
system with a desired feature such as continuity, or which relates the
iterated function system to a known dynamical system, as illustrated in the
following example.

\begin{example}
\label{E:example1} Consider the IFS
\[
\mathcal{F}=\{\mathbb{R},\ f_{0}(x)=tx,\ f_{1}(x)=-tx+1\}
\]
where $t\in\lbrack{\frac{1}{2},1)}$ is a parameter. The attractor of
$A=[0,1].$ Let $M_{1}=[0,\frac{1}{2}],\ M_{2}=(\frac{1}{2},1]$. Then the
masked transformation
\[
T(x)=%
\begin{cases}
\frac{x}{t}, & x\in M_{1},\\
\frac{1-x}{t}, & x\in M_{2}%
\end{cases}
\]
is continuous. This is the well-known one-parameter tent map dynamical system,
see for example \cite[Exercise 2.4.1, p.78]{katok}. We note that, for any
$x\in(0,1)$ there exist a positive integer $n$ such that $T^{k}(x)\in\left[
\frac{2t-1}{2t^{2}},\frac{1}{2t}\right]  $ for all $k\geq n $ (see Fig.
\ref{example1}). As a consequence, if $\Lambda$ denotes the set of masked
addresses of points in $\left[  \frac{2t-1}{2t^{2}},\frac{1}{2t}\right]  $
then the masked address space for $\mathcal{F}$ is
\[
\underset{\text{n times}}{\{\underbrace{111...1}}\sigma|\sigma\in
\Lambda,n=0,1,2,...\}\cup\underset{\text{n times}}{\{2\underbrace{111...1}%
}\sigma|\sigma\in\Lambda,n=0,1,2,...\}\text{.}%
\]
%

\begin{figure}[ptb]%
\centering
\includegraphics[
natheight=14.586900in,
natwidth=14.826800in,
height=2.435in,
width=2.474in
]%
{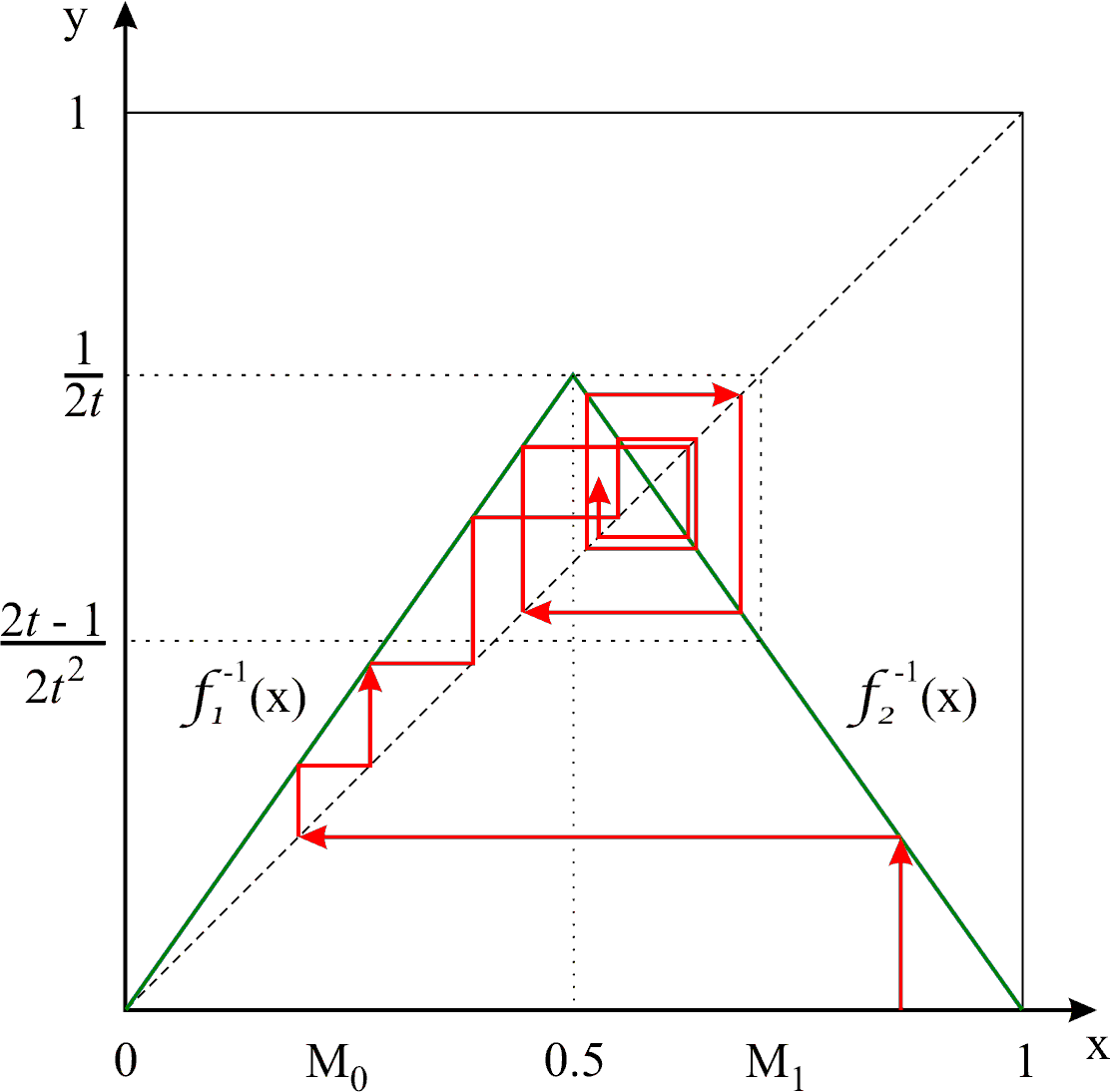}%
\caption{Masked dynamical system on the real interval $[0,1]$. The orbit of
any point eventually enter a trapping interval, $[(2t-1)/2t^{2},1/2t]$. See
Example \ref{E:example1}}%
\label{example1}%
\end{figure}

\end{example}

Theorem \ref{L:disjoint} concerns the relationship between masked address
spaces corresponding to distinct masks. It has an application to packing
multiple images into a single image, as illustrated in Example \ref{expack3}.

\begin{theorem}
\label{L:disjoint} Let $\mathcal{F}$ be a point-fibred iterated function
system with attractor $A$. Let $T:A\rightarrow A$ be a masked dynamical system
for $\mathcal{F}$ corresponding to mask $\mathcal{M}$. Let $\mu$ be a measure
on $A$. Let $\mathcal{M}^{\prime}$ be a mask for $\mathcal{F}$ such that
\[
\mu\left\{  x\in A:\left\{  T^{k}(x)\right\}  _{k=0}^{\infty}\bigcap\left(
\bigcup_{i\in I}\left(  M_{i}\triangle M_{i}^{\prime}\right)  \right)
=\emptyset\right\}  =0.
\]
Let $\tau_{\mathcal{M}}$ and $\tau_{\mathcal{M}^{\prime}}$ be the masked
sections of $\pi$ corresponding to $\mathcal{M}$ and $\mathcal{M}^{\prime}$.
Then $\tau_{\mathcal{M}}(x)\neq\tau_{\mathcal{M}^{\prime}}(x)$ for $\mu
-$almost all $x\in A$.
\end{theorem}

\begin{proof}
If $\tau_{\mathcal{M}}(x)=\tau_{\mathcal{M}^{\prime}}(x)$ it follows that
$T^{k}(x)=T_{\mathcal{M}^{\prime}}^{k}(x)$ for all $k$, where $T_{\mathcal{M}%
^{\prime}}:A\rightarrow A$ is the masked dynamical system corresponding to
$\mathcal{M}^{\prime}$. Then $T_{\mathcal{M}^{\prime}}^{k}(x)\notin\cup_{i\in
I}\left(  M_{i}\triangle M_{i}^{\prime}\right)  $ for all $k$. But this is
impossible for $\mu-$almost all $x\in A$.
\end{proof}

\subsection{(b) The fractal tops method}

"Fractal tops" is the name we use to refer to the mathematics of tops
functions, tops code spaces, tops dynamical systems, and associated fractal
transformations; see for example \cite{monthly, germany}. Here we show that,
in the case where $\mathcal{F}$ is injective, fractal tops arise as a special
case of masked iterated function systems. Specifically, a \textit{tops code
space} is a special case of \textit{masked address space, a tops function} is
a special case of a \textit{section of }$\pi,$ and a \textit{tops dynamical
system} is a special case of a \textit{masked dynamical system}$.$ The
computations associated with fractal tops tend to be less complicated than
those for masked systems.

Define a dictionary ordering (see for example \cite[p.26]{munkres}) on
$I^{\infty}$ as follows: if $\sigma,\omega\in I^{\infty}$, $\sigma\neq\omega,$
then%
\[
\sigma<\omega\text{ iff }\sigma_{k}>\omega_{k},
\]
where $k$ is the least index for which $\sigma_{k}\neq\omega_{k}$, for all
$\sigma,\omega\in I^{\infty}$. With this ordering, every subset of $I^{\infty
}$ possesses a greatest lower bound and a least upper bound. Since
$\pi:I^{\infty}\rightarrow A\subset X$ is continuous and $I^{\infty}$ is
compact, $\pi^{-1}(x)=\{\sigma\in I^{\infty}:\pi(\sigma)=x\}$ possesses a
unique largest element, $\max\pi^{-1}(x)$, for each $x\in A$. This allows us
to define an address space $\Omega\subset I^{\infty}$ for $\mathcal{F}$ by%
\[
\Omega=\{\max\pi^{-1}(x)|x\in A\}\text{.}%
\]
The corresponding section of $\pi$ is%
\[
\tau:A\rightarrow\Omega,\text{ }x\mapsto\max\pi^{-1}(x)\text{.}%
\]
If $\mathcal{F}$ is injective then the tops dynamical system is a masked
dynamical system, corresponding to the mask defined by%
\[
M_{i}=f_{i}(A)\backslash\bigcup_{j=1}^{i-1}f_{j}(A),\ i\in I.
\]
The fact that $\tau\left(  x\right)  $ can be computed from the set $\pi
^{-1}(x)$ without reference to other points on the orbit of $x$ simplifies the
computation of $\tau\left(  x\right)  $ in applications, see for example
\cite{superfractals, monthly}. Note that we can use the orbits of a tops
dynamical system to calculate the top address of any point $x\in A$ according
to $\tau(x)=\sigma_{1}\sigma_{2}...$ where%
\[
\sigma_{k}=\min\{n\in\{1,2,...,N\}:T^{\circ(k-1)}(x)\in f_{n}(A)\}.
\]

\section{\label{appsec}Applications and examples}

\subsection{\label{appsec1}Application to image synthesis}

Here we generalize the technique of color-stealing, introduced in
\cite{barnsley} and implemented for example in \cite[p.65-66]{nickei}.

Define a \textit{picture} to be a function of the form
\[
\mathfrak{P}:D\subset\mathbb{R}^{2}\rightarrow\mathfrak{C}%
\]
where $\mathfrak{C}$ is a color space. The set $D$ is the \textit{domain} of
the picture. We are concerned with situations where $D$ is a subset of an
attractor of an iterated function system.

Let $\mathcal{F}$ be an injective iterated function system with attractor
$A\subset\mathbb{R}^{2}$. Let $\mathcal{F}^{\prime}$ be an iterated function
system with attractor $A^{\prime}\subset\mathbb{R}^{2}$ and the same code
space $I^{\infty}$ as for $\mathcal{F}$. Let $\mathcal{M}$ be a mask for
$\mathcal{F}$ and let $\tau_{\mathcal{M}}$ be the corresponding section of
$\pi^{-1}.$ Let $\pi^{\prime}$ be the coding map for $\mathcal{F}^{\prime}$.
Then we can define a mapping $\Phi$ from the space of pictures on $A^{\prime}$
into the space of pictures on $A$ according to
\[
\Phi_{\mathcal{M}}(\mathfrak{P}^{\prime})=\mathfrak{P}^{\prime}\circ
\pi^{\prime}\circ\tau_{\mathcal{M}}.
\]
We refer to this procedure as color-stealing because colors from the picture
$\mathfrak{P}^{\prime}$ are mapped onto the attractor $A$ to define the new
picture $\mathfrak{P}=\Phi_{\mathcal{M}}(\mathfrak{P}^{\prime})$.

It follows from Theorem \ref{maskbranchthm} (i) that the transformation
$\pi^{\prime}\circ\tau_{\mathcal{M}}$ is continuous at all points $x\in A$
whose orbits lie in $A\backslash$ $\cup_{k=0}^{\infty}\mathcal{F}^{k}%
(\cup_{i\in I}\partial M_{i}),$ where $\partial M_{i}$ denotes the boundary of
$M_{i}$. In some cases, such as those in Example \ref{stealingex}, $\cup
_{k=0}^{\infty}\mathcal{F}^{k}(\cup_{i\in I}\partial M_{i})$ is a set of
Lebesgue measure zero, so the transformation is continuous almost everywhere.
This explains patches of similar colors tends to exist in pictures that are
obtained by color-stealing from real world photos, where patches of similar
colors occur for physical reasons.

\begin{example}
\label{stealingex} Figure \ref{combined} illustrates color-stealing using (i)
fractal tops (left), and (ii) a masked iterated function system (right) that
is not a tops system. The picture $\mathfrak{B}^{\prime}$ from which colors
are stolen is Lena embedded in a black surround, shown in the middle panel. In
(i) $\mathcal{F}^{\prime}$ is an affine iterated function system whose
attractor $A^{\prime}$ is a filled square, the domain of $\mathfrak{B}%
^{\prime}$, such that $\{f_{i}^{\prime}(A^{\prime})\}_{i=1}^{4}$ is a set of
tiles that tile $A^{\prime}$ by rectangles. In (i) $\mathcal{F}$ is the
projective iterated function system $(\mathbb{RP}^{2};f_{1},f_{2},f_{3}%
,f_{4})$, where%
\[
f_{n}(x,y)=(\frac{a_{n}x+b_{n}y+c_{n}}{g_{n}x+h_{n}y+j_{n}},\frac{d_{n}%
x+e_{n}y+k_{n}}{g_{n}x+h_{n}y+j_{n}})\text{ for }n=1,2,3,4\text{;}%
\]%
\[%
\begin{tabular}
[c]{|c|c|c|c|c|c|c|c|c|c|}\hline
$n$ & $a_{n}$ & $b_{n}$ & $c_{n}$ & $d_{n}$ & $e_{n}$ & $k_{n}$ & $g_{n}$ &
$h_{n}$ & $j_{n}$\\\hline
$1$ & $19.05$ & $0.72$ & $1.86$ & $-0.15$ & $16.9$ & $-0.28$ & $5.63$ & $2.01
$ & $20.0$\\
$2$ & $0.2$ & $4.4$ & $7.5$ & $-0.3$ & $-4.4$ & $-10.4$ & $0.2$ & $8.8$ &
$15.4$\\
$3$ & $96.5$ & $35.2$ & $5.8$ & $-131.4$ & $-6.5$ & $19.1$ & $134.8$ & $30.7$
& $7.5$\\
$4$ & $-32.5$ & $5.81$ & $-2.9$ & $122.9$ & $-0.1$ & $-19.9$ & $-128.1$ &
$-24.3$ & $-5.8$\\\hline
\end{tabular}
.
\]
and $\mathcal{F}^{\prime}$ is an affine IFS whose attractor is a filled
square; see also \cite[p.2]{superfractals}. In (ii) the iterated function
system $\mathcal{F}$ and mask $\mathcal{M}$ are the same as in Example
\ref{goldenlennaex}, while $\mathcal{F}^{\prime}$ is a perturbed version of
$\mathcal{G}$ in Example \ref{goldenlennaex}.%
\begin{figure}[ptb]%
\centering
\includegraphics[
natheight=12.239900in,
natwidth=38.479698in,
height=1.7144in,
width=5.3275in
]%
{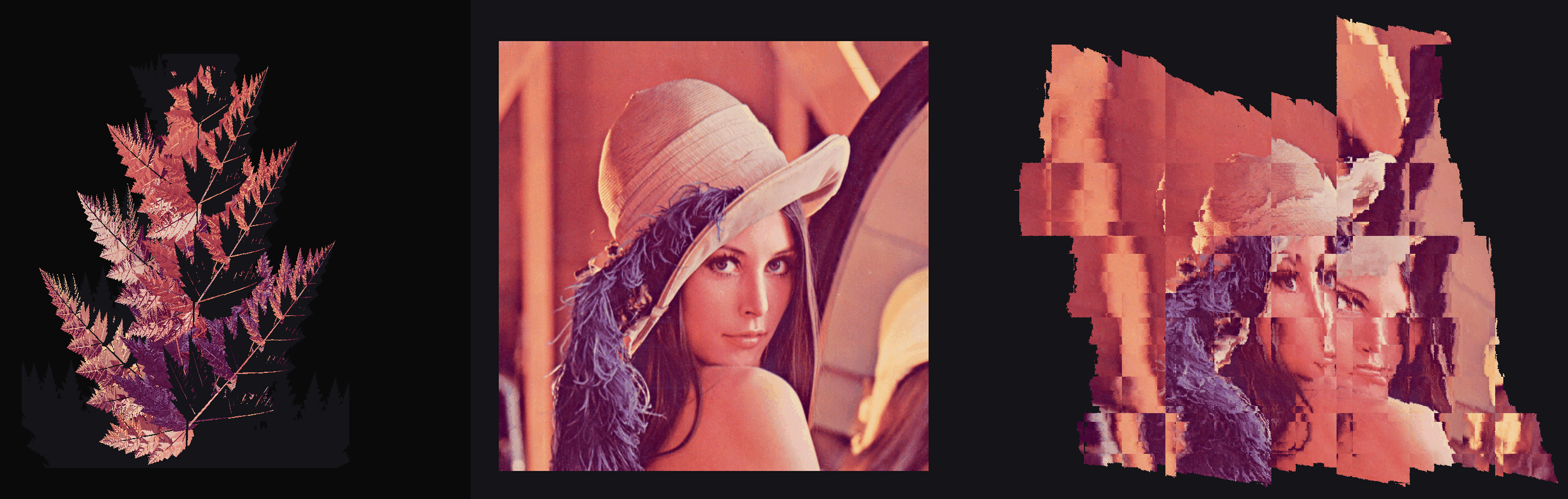}%
\caption{An example of color-stealing using a masked iterated function system
(right) and a tops function\ (left). Although the stolen pictures exhibit many
discontinuites, the transformations, from the original picture of Lena to the
stolen pictures, are continuous almost everywhere. See \ref{stealingex}.}%
\label{combined}%
\end{figure}

\end{example}

\subsection{\label{fhomsec}Fractal homeomorphisms for image beautification,
roughening, and special effects}

Under the conditions of Theorem \ref{ctyfractaltransthm} (ii) the fractal
transformation $T_{\mathcal{FG}}$ is a homeomorphism$.$ Such homeomorphisms
can be applied to pictures to yield new pictures that have the same
topological properties as the original. For example the connectivity
properties of the set defined by a particular color will be preserved, as will
be the property that certain colors lie in an arbitrary neighborhood of a
point. But geometrical properties, such as Hausdorff dimension and
collinearity, may not be preserved and, indeed, may be significantly changed.

Techniques for constructing and computing fractal homeomorphisms using fractal
tops, with projective, affine, and bilinear IFSs, have been discussed in
\cite{barnsley, superfractals, monthly, germany, nickei}. Families of fractal
homeomorphisms, built from such transformations in $\mathbb{R}^{2}$, may be
established by using code space arguments. Typically, the attractors of the
iterated function systems in question are non-overlapping, which simplifies
the proofs: in some situations one only needs to show that the equivalence
classes of addresses agree on certain straight line segments. The resulting
families of transformations are described by a finite sets of real parameters.
These parameters may be adjusted to achieve desired effects such as increased
roughness, or continuous (but non-differentiable) transformation from a
meaningless picture into a meaningful one. Both of these effects are
illustrated in Example \ref{roughex}.

\begin{example}
\label{roughex} Figure \ref{transforms} illustrates three fractal
homeomorphisms of the unit square applied to Lena. All the iterated function
systems involved are constructed using bilinear functions defined as follows.
Let $\square=[0,1]^{2}\subset\mathbb{R}^{2}$ denote the unit square, with
vertices $A=(0,0),B=(1,0),C=(1,1),D=(0,1).$ Let $P,Q,R,S$ denote, in cyclic
order, the successive vertices of a possibly degenerate quadrilateral. We
uniquely define a bilinear function $\mathcal{B}:\mathcal{R}\rightarrow
\mathcal{R}$ such that $\mathcal{B}(ABCD)=PQRS$ by%
\[
\mathcal{B}(x,y)=P+x(Q-P)+y(S-P)+xy(R+P-Q-S).
\]
This transformation acts affinely on any straight line that is parallel to
either the $x$-axis or the $y$-axis.\ For example, if $\mathcal{B}%
|_{AB}:AB\rightarrow PQ$ is the restriction to $AB$ of $\mathcal{B}$ and if
$\mathcal{Q}:\mathbb{R}^{2}\rightarrow\mathbb{R}^{2}$ is the affine function
defined by $\mathcal{Q}(x,y)=P+x(Q-P)+y(S-P),$ then $\mathcal{Q}%
|_{AB}=\mathcal{B}|_{AB}$. Sufficient conditions for a bilinear iterated
function system to be point-fibred are given in \cite{germany}. Each
homeomorphism in Figure \ref{transforms} is generated using a pair of iterated
function systems of the form in Figure \ref{bilinear2}; each such pair has the
same address structure.%
\begin{figure}[ptb]%
\centering
\includegraphics[
natheight=10.893300in,
natwidth=12.586900in,
height=2.4375in,
width=2.4375in
]%
{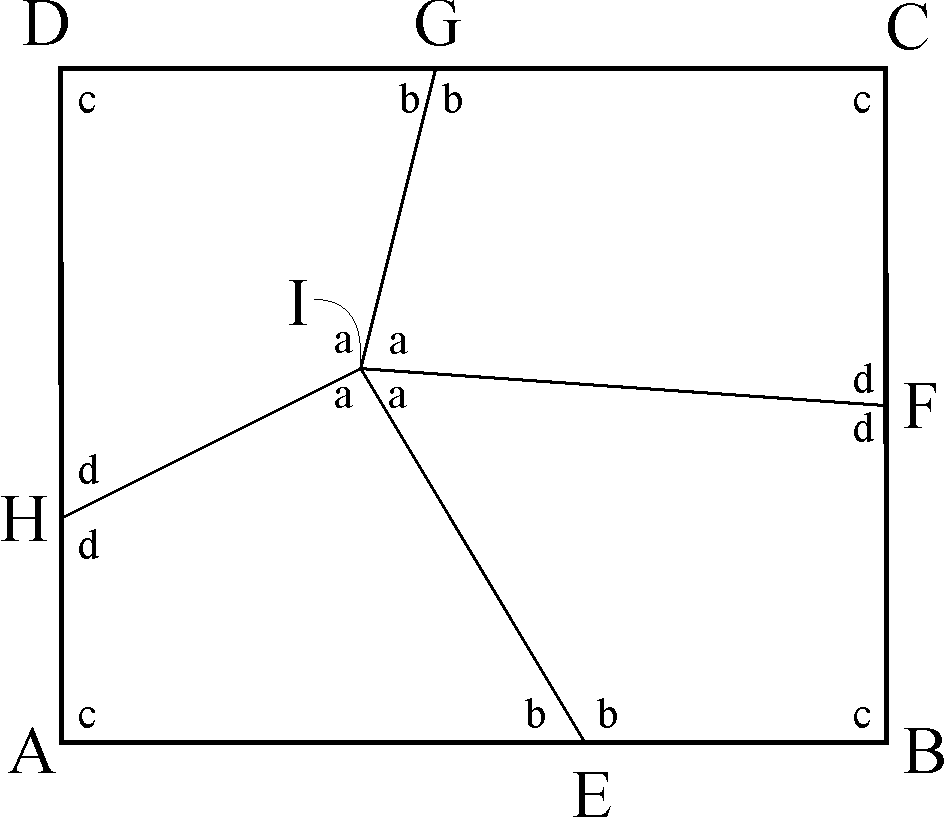}%
\caption{The four quadrilaterals $IEAH$, $IEBF,$ $IGCD$, $IGDH$, define four
bilinear transformations and a corresponding iterated function system whose
attractor is a filled square. Provided that the quadrilaterals are close
enough to quadrants of the square, the iterated function systems is
point-fibred.}%
\label{bilinear2}%
\end{figure}

\end{example}

The following theorem provides practical sufficient conditions for a bilinear
IFS to be hyperbolic.%
\begin{figure}[ptb]%
\centering
\includegraphics[
natheight=6.933100in,
natwidth=20.800200in,
height=1.692in,
width=5.0195in
]%
{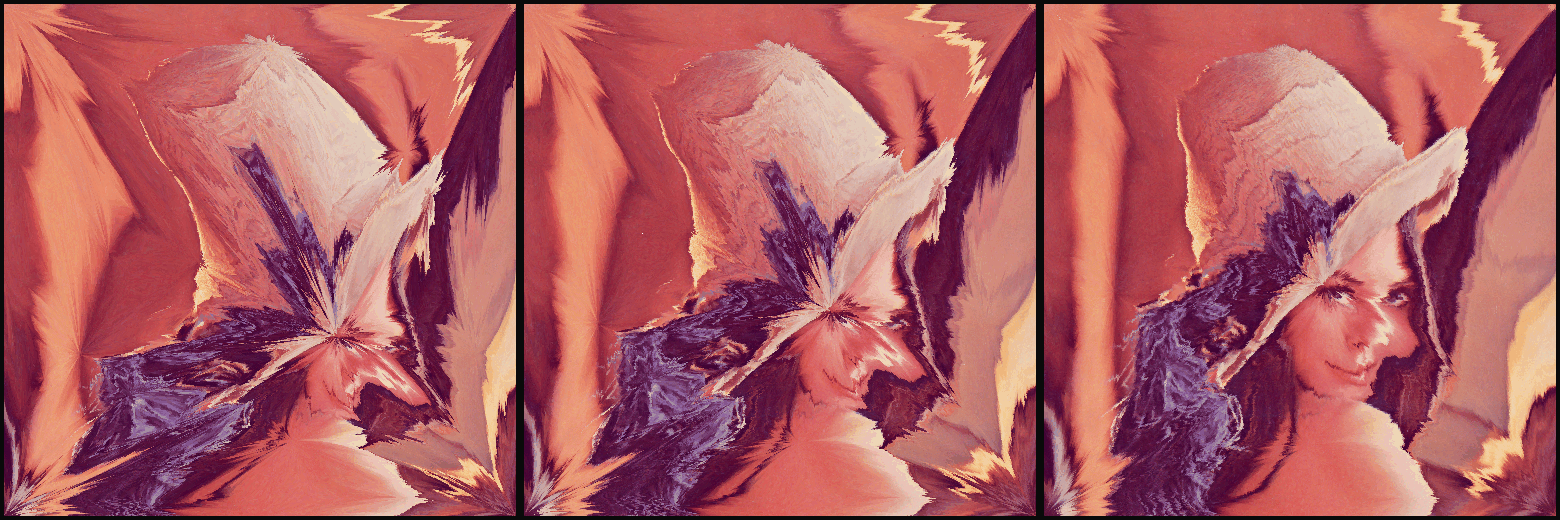}%
\caption{Three fractal homeomorphisms applied to Lena. See Example
\ref{roughex}.}%
\label{transforms}%
\end{figure}

It is more difficult to establish conditions under which pairs of masked
overlapping iterated function systems, built from geometrical functions such
as affines on $\mathbb{R}^{2},$ yield homeomorphisms. In order to establish
that a pair of masked iterated function systems provides a fractal
transformation that is a homeomorphism, it is necessary to establish that the
two masked address spaces agree. Interesting non-trivial cases involve
overlapping iterated function systems, and cannot be reformulated as fractal
tops. A beautiful family of such examples is provided by Theorem \ref{goldthm}.

\begin{theorem}
\label{goldthm} Let $\mathcal{F}=([0,1],f_{1}(x)=ax,f_{2}(x)=by+(1-b)),$ where
$a\geq b>0,$ and $a+b\geq1,$ have mask $M^{(p)}=\{M_{1},M_{2}\}$ where
$M_{1}=[0,p]$, $M_{2}=(p,1],$ and $p\in\lbrack1-b,a]$. Let $\mathcal{G}%
=([0,1],g_{1}(x)=bx,g_{2}(x)=ay+(1-a))$ have mask $M^{(1-p)}$. Then there
exists $p^{\ast}\in\lbrack1-b,a]$ such that the corresponding fractal
transformation $T_{\mathcal{FG}}:[0,1]\rightarrow\lbrack0,1]$ is a
homeomorphism when $p=p^{\ast}.$ The inverse of this homeomorphism is
$T_{\mathcal{GF}}:[0,1]\rightarrow\lbrack0,1]$ defined by associating the mask
$M^{(1-p^{\ast})}$ with $\mathcal{G}$.
\end{theorem}

\begin{proof}
This is an outline. Let $\square=[0,1]^{2}\subset\mathbb{R}^{2}$. Define
affine transformations by
\begin{align*}
W_{-}  &  :\mathbb{R}^{2}\rightarrow\mathbb{R}^{2},(x,y)\mapsto(ax,by),\\
W_{+}  &  :\mathbb{R}^{2}\rightarrow\mathbb{R}^{2},(x,y)\mapsto
(bx+1-b,ay+1-y).
\end{align*}
Let $S_{-}=\{(x,y)\in\mathbb{R}^{2}:x+y\leq1\}$ and $S_{+}=\mathbb{R}%
^{2}\backslash S_{-}$. We consider the dynamical system%
\[
Q:\mathbb{R}^{2}\rightarrow\mathbb{R}^{2},(x,y)\mapsto\left\{
\begin{array}
[c]{c}%
W_{-}^{-1}(x,y)\text{ if }(x,y)\in S_{-},\\
W_{+}^{-1}(x,y)\text{ if }(x,y)\in S_{+}.
\end{array}
\right.
\]
This possesses a "repeller", a compact set $\mathcal{R}\subset\square,$ such
that
\[
\mathcal{R}=Q(\mathcal{R})\text{.}%
\]
In order to define $\mathcal{R}$, we define
\[
W:K\left(  \square\right)  \rightarrow K\left(  \square\right)  ,C\mapsto
(S_{-}\cap W_{-}(C))\cup(\overline{S_{+}}\cap W_{+}(C)),
\]
and
\[
\mathcal{R}=\lim_{k\rightarrow\infty}W^{k}(\square).
\]
$\mathcal{R}$ is well-defined because it is the intersection of a decreasing
sequence of nonempty compact sets. (It is quite easy to see that $\mathcal{R}
$ is the graph of a monotone function from $[0,1]$ onto $[0,1].$) Using
symmetry about the line $x+y=1$ and the contractivity of $W_{-}$ and $W_{+}$
in both the $x$ and $y$ directions it can be proved that $\mathcal{R}$ has the
following properties.

(i) $Q(\mathcal{R})=\mathcal{R};$

(ii) $\mathcal{R}$ is symmetrical about the line $x+y=1;$

(iii) $P_{-}(\mathcal{R})=P_{+}(\mathcal{R})=[0,1],$ where $P_{-}%
:\mathbb{R}^{2}\rightarrow\mathbb{R}$ denotes the projection in the $x$
direction and $P_{+}:\mathbb{R}^{2}\rightarrow\mathbb{R}$ denotes projection
in the $y$ direction;

(iv) there is a continuous, monotone strictly increasing function
$\Phi:[0,1]\rightarrow\lbrack0,1],$ such that $\Phi(0)=0$, $\Phi(1)=1$,
$\Phi(1-\Phi(x))=1-x$ for all $x\in\lbrack0,1],$ and $\mathcal{R}%
=\{(x,\Phi(x)):x\in\lbrack0,1]\}$;

(v) there is a unique $p^{\ast}\in\lbrack1-b,a]$ such that $\Phi(p^{\ast
})=1-p^{\ast};$

(vi) there is a continuous, monotone strictly increasing function,
$\Psi:[0,1]\rightarrow\lbrack0,1],$ such that $\Psi(0)=0$, $\Psi(1)=1$,
$\Psi(1-\Psi(y))=1$ for all $y\in\lbrack0,1]$, and $\mathcal{R}=\{(y,1-\Psi
(1-y)):y\in\lbrack0,1]\}$;

(vii) $\Psi(1-p^{\ast})=p^{\ast};$

(viii) $\Psi(y)=\Phi^{-1}(y)$ for all $y\in\lbrack0,1]$;

(ix) if $p=p^{\ast}$, then the masked dynamical system $T_{\mathcal{F}%
}:[0,1]\rightarrow\lbrack0,1]$ obeys%
\[
T_{\mathcal{F}}(x)=P_{-}(Q(x,\Phi(x)))\text{ for all }x\in\lbrack0,1];
\]

(x) if $p=p^{\ast}$, then the masked dynamical system $T_{\mathcal{G}%
}:[0,1]\rightarrow\lbrack0,1]$ obeys%
\[
T_{\mathcal{G}}(y)=P_{+}(Q(\Psi(y),y))\text{ for all }y\in\lbrack0,1].
\]

These statements imply the theorem.
\end{proof}

\begin{remark}
Clearly essentially the same result and proof applies for any analogous pair
of overlapping strictly increasing functions on $[0,1]$.
\end{remark}

Figure \ref{goldenimage} illustrates the "repeller". It is a subset of the
attractor of the iterated function system $(\square;(ax,by),(bx+1-b,ay+1-a)),$
and may be used to compute $p^{\ast}$ as illustrated in Figure \ref{figure}.%
\begin{figure}[ptb]%
\centering
\includegraphics[
natheight=13.652900in,
natwidth=13.652900in,
height=2.435in,
width=2.435in
]%
{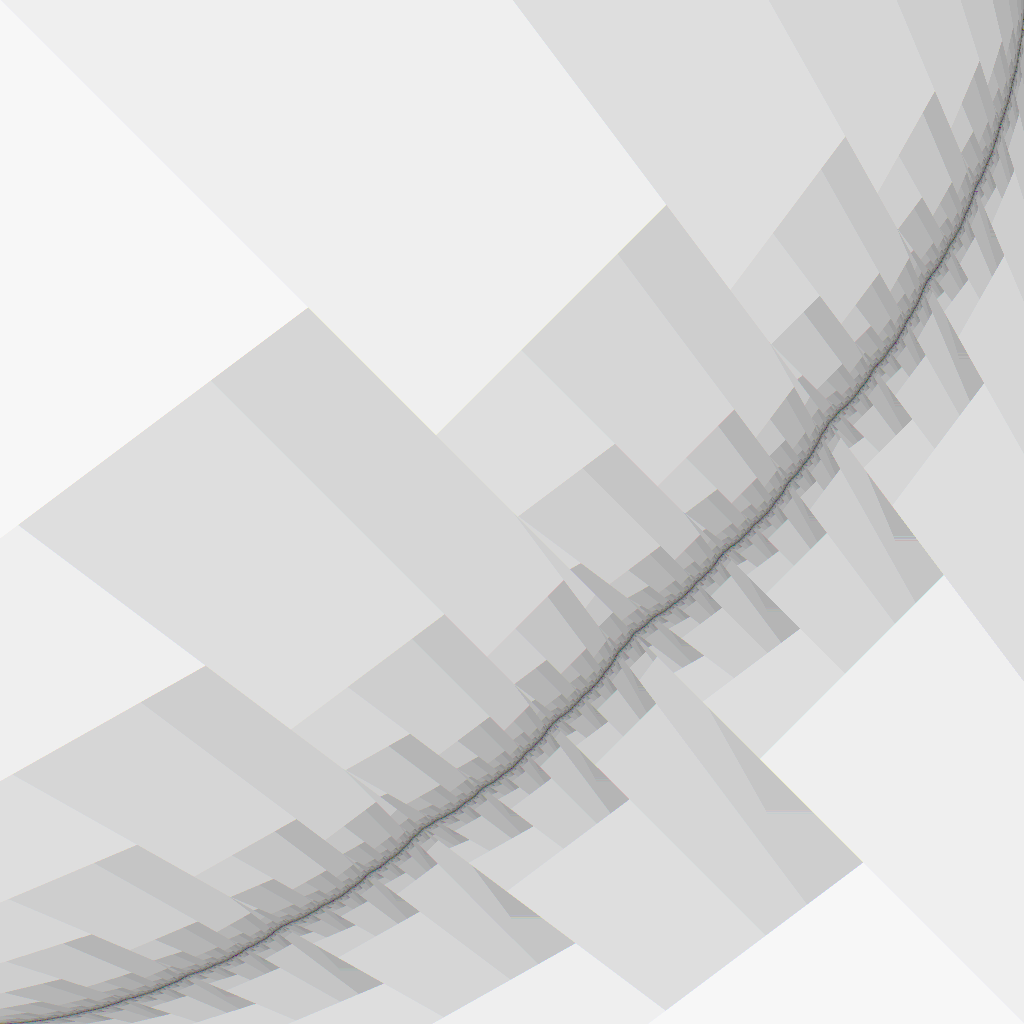}%
\caption{This shows approximations, in shades of grey, to the "repeller"
$\mathcal{R}$ of the dynamical system $Q:\mathbb{R}^{2}\rightarrow
\mathbb{R}^{2}$ described in the outline proof of Theorem \ref{goldthm}. An
escape-time algorithm, similar to the ones discussed in \cite[Ch.
7]{barnsleyFE}, was used to make this image, with $a=2/3$ and $b=1/2.$}%
\label{goldenimage}%
\end{figure}

An example of a fractal transformation, arising from a masked pair of
overlapping affine iterated function systems, is given in Example
\ref{goldenlennaex}. The resulting homeomorphism with $a=\frac{2}{3}$ and
$b=\frac{1}{2}$ yields a picture of Lena with extra large eyes, Figure
\ref{goldenlenna}.

\begin{example}
\label{goldenlennaex} Let $\mathcal{H}_{p,q}:=(X;h_{1},h_{2},h_{3},h_{4})$ be
the family of affine iterated function systems defined by $X\mathbb{=\{}%
(x,y)\in\mathbb{R}^{2}|0\leq x,y\leq1\},$ $r,s\in(0,1),$
\begin{align*}
h_{1}(x,y)  &  =(rx,ry),h_{2}(x,y)=(sx+1-s,ry),\\
h_{3}(x,y)  &  =(sx+1-s,sy+1-s),h_{4}(x,y)=(rx,sy+1-s).
\end{align*}
Let $\mathcal{F}=\mathcal{H}_{\frac{2}{3},\frac{1}{2}}$ and, for $p\in
\lbrack\frac{1}{2},\frac{2}{3}],$ let $M_{p}$ be the mask$\{M_{1},M_{2}%
,M_{3},M_{4}\}$ where
\begin{align*}
M_{1}  &  =\{(x,y)\in X:x\leq p,y\leq p\},M_{2}=\{(x,y)\in X:x>p,y\leq p\},\\
M_{3}  &  =\{(x,y)\in X:x>p,y>p\},M_{4}=\{(x,y)\in X:x\leq p,y>p\}.
\end{align*}
Let $\mathcal{G}=\mathcal{H}_{\frac{1}{2},\frac{2}{3}}.$ Then, by Theorem
\ref{goldthm}, there is a value of $p=$ $p^{\ast}\in(\frac{1}{2},\frac{2}%
{3}),$ such that the condition in Theorem \ref{ctyfractaltransthm} (ii) holds,
and if $p=p^{\ast}$ then $T_{\mathcal{FG}}:X\rightarrow X$ then is a
homeomorphism. The value of $p^{\ast}$ is%
\[
p^{\ast}=\max\{p\in(\frac{1}{2},\frac{2}{3}):(p,1-p)\in A^{\ast}\}
\]
where $A^{\ast}\subset X$ is the attractor of $\mathcal{H}:=(X;(\frac{2}%
{3}x,\frac{1}{2}y),(\frac{1}{2}x+\frac{1}{2},\frac{2}{3}y+\frac{1}{3}))$
illustrated in Figure \ref{figure}. Experimentally we find $p^{\ast}%
\doteq0.618$ which is used to compute Figure \ref{goldenlenna}$.$ The original
Lena was overlayed on a black background, as in Figure \ref{combined}. The
transformed picture also had a black background that has been omitted here.
\end{example}

%

\begin{figure}[ptb]%
\centering
\includegraphics[
natheight=9.227000in,
natwidth=9.133200in,
height=2.435in,
width=2.4109in
]%
{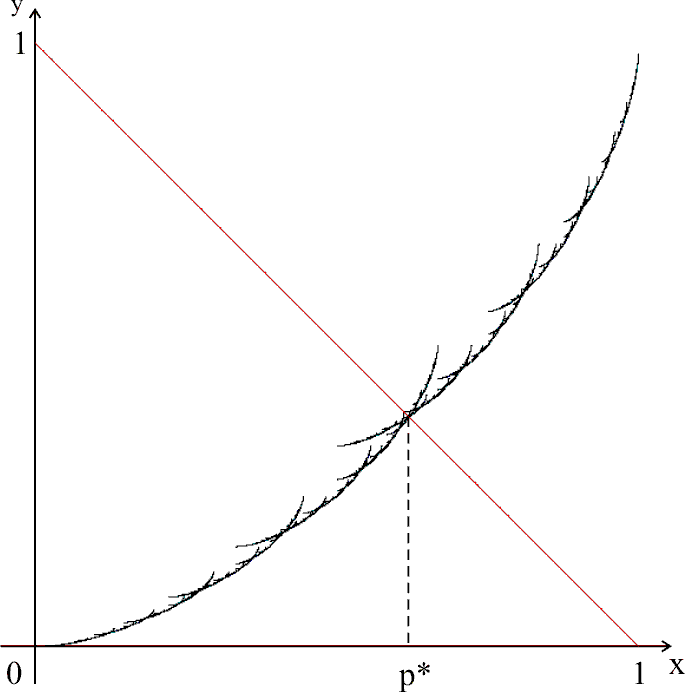}%
\caption{The value $x=p^{\ast}$ in Example \ref{goldenlennaex} is the maximum
$x$ such that the line $y=1-x$ (red) meets the attractor $A^{\ast}$ (black) of
$(\square;(\frac{2}{3}x,\frac{1}{2}y),(\frac{1}{2}x+\frac{1}{2},\frac{2}%
{3}y+\frac{1}{3}))$. }%
\label{figure}%
\end{figure}

\subsection{\label{filtsec} Application to image filtering}

Here we restrict attention to iterated function systems defined on
$X=[0,1]^{2}\subset$ $\mathbb{R}^{2}$. We are concerned with fractal
transformations from $[0,1]^{2}$ to itself, associated with a pair of iterated
function systems of the form%
\[
\mathcal{F}=([0,1]^{2},f_{1},f_{2},...,f_{N})\text{, }\mathcal{G}%
=([0,1]^{2},g_{1},g_{2},...,g_{N}).
\]
We suppose that $A_{\mathcal{F}}=A_{\mathcal{G}}=[0,1]^{2}$. In applications
to digital imaging, $[0,1]^{2}$ is discretized. Here we refer to the the
locations of pixels as discretized coordinates. We suppose that discretized
versions of $A_{\mathcal{F}}$ and $A_{\mathcal{G}}$ have resolutions
$r_{\mathcal{F}}$ and $r_{\mathcal{G}}$ respectively.

Let $P_{\mathcal{F}}:A_{\mathcal{F}}\rightarrow A_{\mathcal{F}}$ be a
projection operator, namely a function with the property $P_{\mathcal{F}}\circ
P_{\mathcal{F}}(x)=P_{\mathcal{F}}(x)$ for all $x\in A_{\mathcal{F}}$. For
example,
\[
P_{\mathcal{F}}(x)=x\text{ discretized to resolution }r_{\mathcal{F}},\text{
for all }x\in\lbrack0,1]^{2}\text{.}%
\]
If $T_{\mathcal{FG}}:A_{\mathcal{F}}\rightarrow A_{\mathcal{G}}$ is a
homeomorphism between the non-discretized spaces then
\[
P_{\mathcal{G}}:=T_{\mathcal{FG}}\circ P_{\mathcal{F}}\circ T_{\mathcal{GF}%
}:A_{\mathcal{G}}\rightarrow A_{\mathcal{G}}%
\]
is also a projection operator because $T_{\mathcal{FG}}\circ T_{\mathcal{GF}%
}=i_{\mathcal{G}},$ the identity on $A_{\mathcal{G}},$ whence $P_{\mathcal{G}%
}\circ P_{\mathcal{G}}=T_{\mathcal{FG}}\circ P_{\mathcal{F}}\circ
T_{\mathcal{GF}}\circ T_{\mathcal{FG}}\circ P_{\mathcal{F}}\circ
T_{\mathcal{GF}}=T_{\mathcal{FG}}\circ P_{\mathcal{F}}\circ i_{\mathcal{F}%
}\circ P_{\mathcal{F}}\circ T_{\mathcal{GF}}=i_{\mathcal{G}}$.

A trivial but instructive example is provided by choosing $\mathcal{F}$ and
$\mathcal{G}$ to be the same, with $N=4$ and%
\begin{align*}
f_{1}(x,y)  &  =(0.5x,0.5y),f_{2}(x,y)=(0.5x+0.5,0.5y),\\
f_{3}(x,y)  &  =(0.5x+0.5,0.5y+0.5),f_{4}(x,y)=(0.5x,0.5y+0.5).
\end{align*}
If $r_{\mathcal{F}}=r_{\mathcal{G}}/2$ then $P_{\mathcal{G}}$ is the filter
that corresponds to downsampling followed by doubling the width and height of
each pixel.

\begin{example}
\label{lenaex}Let $X\mathbb{=\{}(x,y)\in\mathbb{R}^{2}|0\leq x,y\leq1\},$
$p\in(0,1)$, $q=1-p,$ and
\begin{align*}
h_{1}(x,y)  &  =(px,py),h_{2}(x,y)=(qx+p,py),\\
h_{3}(x,y)  &  =(qx+p,qy+p),h_{4}(x,y)=(px,qy+p).
\end{align*}
The family of IFSs $\mathcal{H}_{p}:=(X;h_{1},h_{2},h_{3},h_{4})$ has
attractor $X$ and address structure that is independent of $p\in(0,1).$ If we
set $\mathcal{F}=\mathcal{H}_{0.5}$ and $\mathcal{G}=\mathcal{H}_{0.6}$ then
$T_{\mathcal{FG}}:X\mathbb{\rightarrow}X$ is a homeomorphism. The result of
applying the fractal homeomorphism $T_{\mathcal{FG}}$ to a digital
($512\times512$) picture of Lena (left) is illustrated in the middle panel of
Figure \ref{lena}. In effect the middle image is obtained by composing a
projection $P_{\mathcal{F}}$, onto a $512\times512$ pixel grid, with
$T_{\mathcal{FG}}$. The image on the right is the result of applying
$T_{\mathcal{GF}}$ to the middle image. that is, the right-hand image is the
result of applying the projection operator, that we may refer to as a "fractal
filter", $T_{\mathcal{GF}}\circ P_{\mathcal{F}}\circ$ $T_{\mathcal{FG}}$ to
the original Lena.
\begin{figure}[ptb]%
\centering
\includegraphics[
natheight=6.826900in,
natwidth=20.800200in,
height=1.7684in,
width=5.3275in
]%
{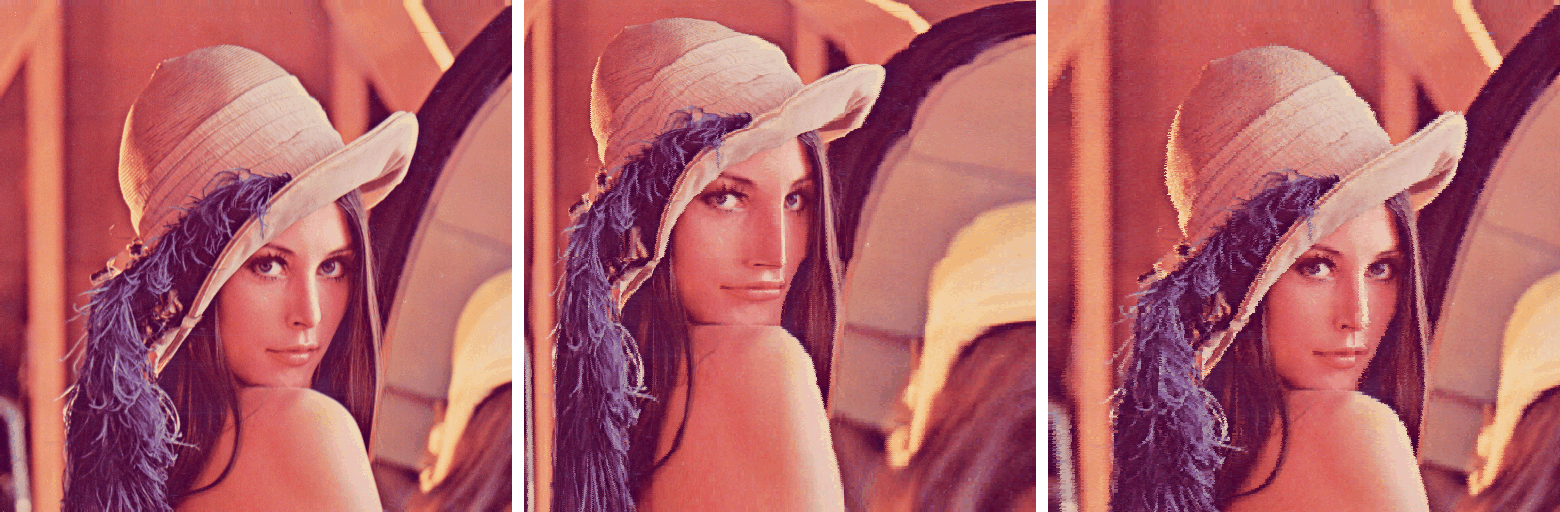}%
\caption{Lena before (left) and after (middle) a fractal homeomorphism has
been applied. The image on the right is the result of applying the
corresponding fractal filter. See Example \ref{lenaex}.}%
\label{lena}%
\end{figure}

\end{example}

\subsection{\label{packsec}Application to packing multiple images into a
single image}

Suppose we have a collection of masks $\{\mathcal{M}_{j}\}$ on $A$ such that
the conditions of Lemma \ref{L:disjoint} are holds true for any pair of masks.
Let the second iterated function system $\mathcal{F}^{\prime}$ be such that
almost all points of the attractor $A^{\prime}$ have a unique address. Thus by
Lemma \ref{L:disjoint}
\[
\pi^{\prime}\circ\tau_{\mathcal{M}_{j}}(x)\neq\pi^{\prime}\circ\tau
_{\mathcal{M}_{k}}(x)
\]
for almost all $x\in A$ and $j\neq k$. Therefore we can consider the
collection of pictures $\Phi_{\mathcal{M}_{j}}(\mathfrak{B}^{\prime})$ as an
almost disjoint fractal decomposition of the picture $\mathfrak{B}^{\prime}$.

This leads us to the following trick. By means of fractal transformations we
map different pictures to different almost disjoint components of
$\mathfrak{B}^{\prime}$. By inverting the transformations we retrieve
(approximations to) the original pictures, as illustrated in the following example.

\begin{example}
\label{expack3} See Figure \ref{allmasked}. Let
\[
\mathcal{F}=(\square,f_{1}(x,y)=(y,0.6(1-x)),f_{2}(x,y)=(y,0.4+0.6x))
\]
and
\[
\mathcal{G}=(\square,f_{1}(x,y)=(y,0.5(1-x)),f_{2}(x,y)=(y,0.5+0.5x)).
\]
We define a family of masks for $\mathcal{F}$ by
\[
\mathcal{M}_{p}=\{M_{1}=\{(x,y)\in\square:x\leq p\},M_{2}=\{(x,y)\in
\square:x>p\}\},p\in\lbrack0.4,0.6].
\]
$\square$ is the attractor of both systems. We denote the corresponding masked
fractal transformation $T_{\mathcal{FG}}:\square\rightarrow\square$ by
$T_{\mathcal{FG}}^{(p)}:\square\rightarrow\square,$ for $p\in\lbrack0.4,0.6]$.
Then $T_{\mathcal{FG}}^{(p)}$ is injective and invertible on its range,
$T_{\mathcal{FG}}^{(p)}(\square)$. Moreover, from Theorem \ref{L:disjoint} it
follows that $\lambda(T_{\mathcal{FG}}^{(p)}(\square)\cap T_{\mathcal{FG}%
}^{(q)}(\square))=0$ for all $p\neq q$, where $\lambda$ is Lebesgue measure.
Let $\mathfrak{P}$ and $\mathfrak{Q}$ be two pictures, each supported on
$\square$. Then $T_{\mathcal{FG}}^{(p)}(\mathfrak{P)}=\mathfrak{P}%
\circ(T_{\mathcal{FG}}^{(p)})^{-1}$ is a picture supported on $T_{\mathcal{FG}%
}^{(p)}(\square)$ and $T_{\mathcal{FG}}^{(q)}(\mathfrak{Q)}=\mathfrak{Q}%
\circ(T_{\mathcal{FG}}^{(p)})^{-1}$ is a picture supported on $T_{\mathcal{FG}%
}^{(q)}(\square)$. We choose $p=0.44,$ $q=0.56,$ $\mathfrak{P}=$Lena and
$\mathfrak{Q}=$Inverted-Lena, where both pictures are $512\times512$. The
digitized combined picture $\mathfrak{R}:=T_{\mathcal{FG}}^{(p)}%
(\mathfrak{P)}\cup T_{\mathcal{FG}}^{(q)}(\mathfrak{Q)}$, also of resolution
$512\times512$, is shown in the middle panel of Figure \ref{allmasked}. Pixels
which correspond to points in $T_{\mathcal{FG}}^{(p)}(\square)$ both
$T_{\mathcal{FG}}^{(q)}(\square)$ are colored white. The left-hand panel in
Figure \ref{allmasked} illustrates the picture $\mathfrak{R\circ
}T_{\mathcal{FG}}^{(p)}$ and the right-hand panel shows $\mathfrak{R\circ
}T_{\mathcal{FG}}^{(q)}$ . Hence we can "store" the two pictureswe have that
for each distinct choice of $p\in\lbrack0.4,0.6]$ the ranges of the
$\mathcal{M}_{0.44}$ and $\mathcal{M}_{0.56}$ respectively.%
\begin{figure}[ptb]%
\centering
\includegraphics[
natheight=6.826900in,
natwidth=20.800200in,
height=1.8406in,
width=5.5492in
]%
{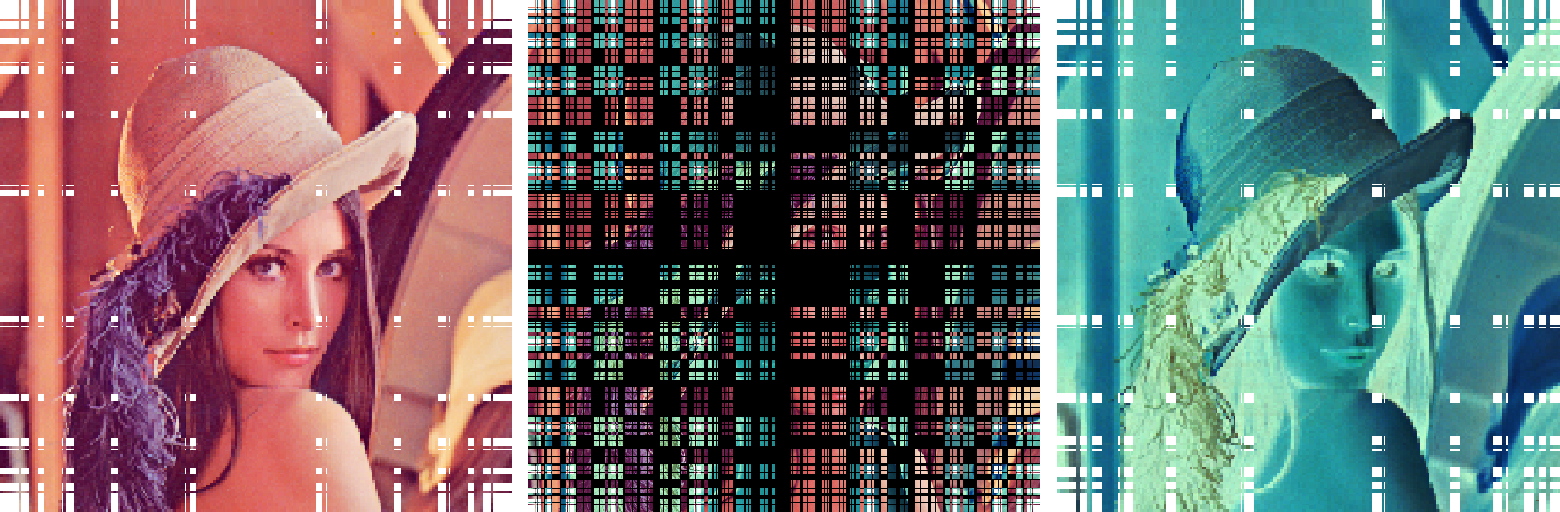}%
\caption{The pictures on the right and left were obtained by masking; the
"encoding" is shown in the middle image. In this example each of the image is
of resolution approximately $512\times512$. See Example \ref{expack3}.}%
\label{allmasked}%
\end{figure}

\end{example}

\subsection{\label{meassec}Measure theory image packing}

Here we describe a different method for encoding several images within a
single image. The method strives to create a single image that is
simultaneously "homeomorphic", under different fractal transformations, to
several different given images. An example of such a single "encoded" image is
shown in the middle panel of Figure \ref{all}; two different fractal
homeomorphisms applied to it (but in the digital realm) yield the images on
the right and the left. The underlying idea is that different invariant
measures, belonging to the same iterated function system, but associated with
different probabilities, are mutually "singular continuous", and are
concentrated on different sets, albeit they have the same support. The method
uses the chaos game algorithm, both to encode and decode.

Consider for example the IFS $\mathcal{F}=([0,1];f_{1}(x)=0.5x,f_{2}%
(x)=0.5x+0.5)$. If we associate probabilities $p_{1}=0.5$ and $p_{2}=0.5$ with
$f_{1}$ and $f_{2}$ respectively, then the associated Markov operator has
invariant probability measure $\mu$ equal to uniform Lebesgue measure
supported on $[0,1]$, concentrated on the set of points whose binary
expansions contain equal proportions of zeros and ones. On the other hand, if
we associate probabilities $\widetilde{p}_{1}=0.1$ and $\widetilde{p}%
_{2}=0.9,$ then the associated Markov operator has invariant probability
measure $\widetilde{\mu}$ that is also supported on $[0,1]$, but singular
continuous and concentrated on the set of points whose binary expansions
contain nine times as many ones as zeros. Hence, if the chaos game is applied
to $\mathcal{F}$ using the first set of probabilities, the points of the
resulting random orbit will tend to be disjoint from those obtained by of
applying the chaos game using the second set of probabilities. The precise
manner in which the orbits of the two systems concentrate is governed by
convergence rates associated with the central limit theorem. The same idea can
be applied to "store" several images in a single image E: a stored image is
retrieved by applying the appropriate fractal transformation to E.
\begin{figure}[ptb]%
\centering
\includegraphics[
natheight=6.826900in,
natwidth=20.800200in,
height=1.8406in,
width=5.5492in
]%
{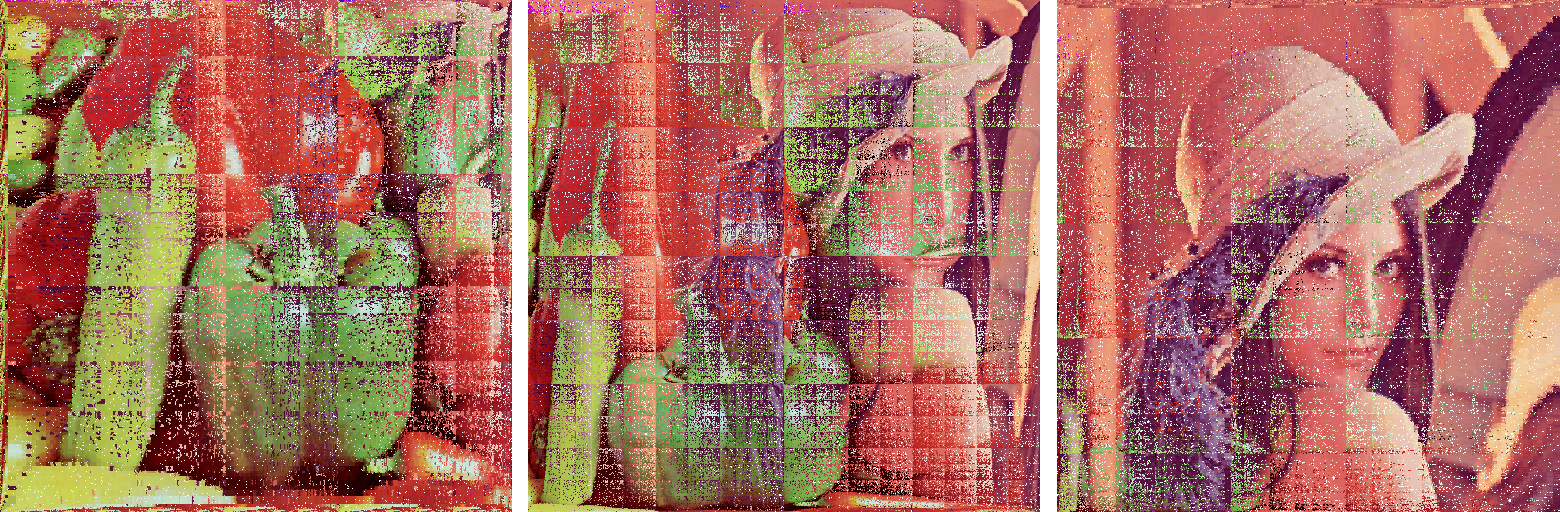}%
\caption{Example of singular measure encoding. Each ordered pair of these
images is related by a fractal homeomorphism.}%
\label{all}%
\end{figure}

We describe the method by means of an example.

\begin{example}
\label{expack4} We use three iterated function systems:%
\begin{align*}
\mathcal{F}  &  =(\square,f_{1},f_{2},f_{3},f_{4}),\text{ }\mathcal{G}%
=(\square,g_{1},g_{2},g_{3},g_{4})\text{, }\mathcal{H}=(\square,h_{1}%
,h_{2},h_{3},h_{4})\text{ where}\\
f_{1}(x,y)  &  =(0.66x,0.34y),\text{ }f_{2}(x,y)=(0.34x+0.66,0.34y),\\
f_{3}(x,y)  &  =(0.34x+0.66,0.66y+0.34),\text{ }f_{4}%
(x,y)=(0.66x,0.66y+0.34);\\
g_{1}(x,y)  &  =(0.34x,0.66y),\text{ }g_{2}(x,y)=(0.66x+0.34,0.66y),\\
g_{3}(x,y)  &  =(0.66x+0.34,0.34y+0.66),\text{ }g_{4}%
(x,y)=(0.34x,0.34y+0.66);\\
h_{1}(x,y)  &  =(0.5x,0.5y),\text{ }h_{2}(x,y)=(0.5x+0.5,0.5y),\\
h_{3}(x,y)  &  =(0.5x+0.5,0.5y+0.5),\text{ }h_{4}(x,y)=(0.5x,0.5y+0.5).
\end{align*}
Then $\mathcal{F}$ is associated with probabilies $P=\{p_{i}=area\left(
f_{i}(\square\right)  ):$ $i=1,2,3,4\}$ Similarly the iterated function system
$\mathcal{G}$ is associated with probabilities $\widetilde{P}=\{\widetilde{p}%
_{i}=area(g_{i}(\square)):i=1,2,3,4\}.$

The goal is to "store" two standard digital color images, Pepper and Lena,
each $512\times512,$ in a single color image $E,$ also $512\times512$ . The
image $E$ is supported on $\square$ and associated with two probability
measures, $\mu_{\mathcal{H}}$ and $\widetilde{\mu}_{\mathcal{H}},$ invariant
under $\mathcal{H}$ with probabilities $P$ and $\widetilde{P}$ respectively.

(1) In order to "encode" Pepper, we run a coupled chaos game algorithm with
$\mathcal{F}$ associated with Pepper, supported on a copy of $\square,$ and
$\mathcal{H}$ associated with $E,$ with (supposedly) i.i.d. probabilities $P$;
that is, we compute a random sequence of points $\{(X_{k},Z_{k})\in
\square\times\square:k=0,1,2,...K\}$ where $K=10^{6},$ where $X_{0}%
=(0,0)\in\square$ (associated with Pepper) and $Z_{0}=(0,0)\in\square$
(associated with $E$) and
\[
X_{k}=f_{\sigma_{k}}(X_{k-1}),Z_{k}=h_{\sigma_{k}}(Z_{k-1})\text{ for
}k=1,2,...,K,
\]
where $\sigma_{k}=i$ with probability $p_{i}$. At each step the pixel
containing $Z_{k}$ in the (initially blank) image $E$ is plotted in the color
of Pepper at the pixel containing $X_{k}$. At the end of this process $E$
consists of an "encoded" version of Pepper.

(2) In order to "encode" Lena, we again run a coupled chaos game algorithm
with $\mathcal{G}$ associated with Lena, supported on a copy of $\square,$ and
$\mathcal{H}$ associated with $E$ (already "painted" with an encoding of
Pepper) with probabilities $\widetilde{P}$; that is, we compute a sequence of
points $\{(Y_{k},\widetilde{Z}_{k})\in\square\times\square:k=0,1,2,...K\}$
where $K=500,000,$ where $Y_{0}=(0,0)\in\square$ (associated with Lena) and
$Z_{0}=(0,0)\in\square$ (associated with $E$) and
\[
Y_{k}=f_{\sigma_{k}}(Y_{k-1}),\widetilde{Z}_{k}=h_{\sigma_{k}}(\widetilde{Z}%
_{k-1})\text{ for }k=1,2,...,K,
\]
where $\sigma_{k}=i$ with probability $\widetilde{p}_{i}$. At each step the
pixel containing $\widetilde{Z}_{k}$ in the image $E$ is plotted in the color
of Lena at the pixel containing $Y_{k}$. (The pixel is overwritten by the
latest colour value.) We have used half as many iterations in the encoding of
Lena as we did for Pepper, because a proportion of the points that correspond
to Pepper are overwritten by points corresponding to Lena.

An image $E$ that is a realization of steps (1) and (2) is shown in the
central panel of Figure \ref{all}. The image $E$ is approximately homeomorphic
to both of the images Pepper and Lena, under the fractal transformations
$T_{\mathcal{FH}}$ and $T_{\mathcal{GH}}$ respectively, that is%
\[
Pepper\doteq T_{\mathcal{FH}}(E)\text{ and }Lena\doteq T_{\mathcal{GH}}(E).
\]
In practice, to obtain the decoded images, shown on the left and right hand
sides of Figure \ref{all}, we use the chaos game algorithm again, as follows.

(3) In order to "decode" Pepper, we run a coupled chaos game algorithm, with
probabilities $P,$ with $\mathcal{F}$ associated with an image (initially
blank), supported on a copy of $\square,$ and $\mathcal{H}$ associated with
$E$ (now encoding both Pepper and Lena). That is, we compute a random sequence
of points $\{(X_{k},Z_{k})\in\square\times\square:k=0,1,2,...K\} $ where
$K=10^{6},$ where $X_{0}=(0,0)\in\square$ (associated with Pepper) and
$Z_{0}=(0,0)\in\square$ (associated with $E$) and
\[
X_{k}=f_{\sigma_{k}}(X_{k-1}),Z_{k}=h_{\sigma_{k}}(Z_{k-1})\text{ for
}k=1,2,...,K,
\]
where $\sigma_{k}=i$ with probability $p_{i}$. At each step the pixel
containing $X_{k}$ in the (initially blank) copy of $\square$ is plotted in
the color of $E$ at the pixel containing $Z_{k}$. The result of such a
decoding, starting from the encoded $E$ illustrated in the middle panel, is
shown in the left panel in Figure \ref{all}.

(4) In order to decode Lena, we run a coupled chaos game algorithm, with
probabilities $\widetilde{P},$ with $\mathcal{G}$ associated with an image
(initially blank but to become the decoded image), supported on a copy of
$\square,$ and $\mathcal{H}$ associated with $E$. That is, we compute a random
orbit $\{(Y_{k},\widetilde{Z}_{k})\in\square\times\square:k=0,1,2,...K\}$
where $K=10^{6},$ where $Y_{0}=(0,0)\in\square$ (associated with Pepper) and
$\widetilde{Z}_{0}=(0,0)\in\square$ (associated with $E$) and
\[
Y_{k}=f_{\sigma_{k}}(X_{k-1}),\widetilde{Z}_{k}=h_{\sigma_{k}}(\widetilde{Z}%
_{k-1})\text{ for }k=1,2,...,K,
\]
where $\sigma_{k}=i$ with probability $\widetilde{p}_{i}$. At each step the
pixel containing $Y_{k}$ in the (initially blank) copy of $\square$ is plotted
in the color of $E$ at the pixel containing $\widetilde{Z}_{k}$. The result of
following this decoding algorithm, starting from the encoded $E$ in the middle
panel, is shown in the left panel in Figure \ref{all}.
\end{example}

\section*{Acknowledgments}

We thank Louisa Barnsley for help with the illustrations.

\end{document}